\documentclass[11pt]{amsart}
\pdfoutput=1
\usepackage{latexsym, graphicx, epsfig, amsmath, amsfonts,amssymb,subfigure,amsthm}
\usepackage{amsmath}
\usepackage{amsfonts}
\usepackage{amssymb}
\usepackage{algorithm}
\usepackage{algorithmic}
\usepackage{bm}
\usepackage{xcolor}
\usepackage{float}
\usepackage{graphicx}
\usepackage{epstopdf,mathtools,bbm}
\usepackage[left=1.25in,right=1.25in,top=1in]{geometry}

\def\R{\mathbb{R}}
\def\E{\mathbb{E}}
\newcounter{Rownumber} 
\DeclareMathOperator*{\argmin}{arg\,min}
\DeclareMathOperator*{\argmax}{arg\,max}

\newcommand{\bs}[1]{\boldsymbol{#1}}
\newcommand{\N}{\mathbbm{N}}
\newcommand{\calA}{\mathcal{A}}

\newtheorem{lemma}{Lemma}[section]
\newtheorem{theorem}{Theorem}[section]
\newtheorem{definition}{Definition}[section]
\newtheorem{corollary}{Corollary}[section]

\title{Weighted approximate Fekete points: sampling for least-squares polynomial approximation}

\author{Ling Guo}
\thanks{Department of Mathematics, Shanghai Normal University, Shanghai, China. Email: lguo@shnu.edu.cn. L. Guo is partially supported by NSFC (11101287, 11671265).}

\author{Akil Narayan}
\thanks{Department of Mathematics and Scientific Computing and Imaging Institute,
	University of Utah, Salt Lake City, UT 84112. Email: akil@sci.utah.edu. A. Narayan is partially supported by AFOSR FA9550-15-1-0467, and DARPA EQUiPS N660011524053.}

\author{Liang Yan}
\thanks{Department of Mathematics, Southeast University, Nanjing, China. Email: yanliang@seu.edu.cn.}

\author{Tao Zhou}
\thanks{LSEC, Institute of Computational Mathematics, Academy of Mathematics and Systems
Science, Chinese Academy of Sciences, Beijing 100190, China. Email: tzhou@lsec.cc.ac.cn. T. Zhou is partially supported the National Natural Science Foundation of China (Award Nos. 91630312, 91630203, and 11571351) and NCMIS.}

\begin{document}

\maketitle

\begin{abstract}
  We propose and analyze a weighted greedy scheme for computing deterministic sample configurations in multidimensional space for performing least-squares polynomial approximations on $L^2$ spaces weighted by a probability density function. Our procedure is a particular weighted version of the approximate Fekete points method, with the weight function chosen as the (inverse) Christoffel function. Our procedure has theoretical advantages: when linear systems with optimal condition number exist, the procedure finds them. In the one-dimensional setting with any density function, our greedy procedure almost always generates optimally-conditioned linear systems. Our method also has practical advantages: our procedure is impartial to compactness of the domain of approximation, and uses only pivoted linear algebraic routines. We show through numerous examples that our sampling design outperforms competing randomized and deterministic designs when the domain is both low and high dimensional.
%  Our procedure allows us to utilize connections between matrix determinants and condition numbers. When optimal sample sets exist, our greedy procedure
%
%  in one dimension Gauss quadrature rules are produced and result in optimally conditioned linear systems. In multiple dimensions
%
%Least-squares regression for constructing polynomial chaos expansions is a hot topic recently. Much attention has been paid to the sampling strategy in the least-squares approach.  In this work, we propose a quasi-optimal sampling strategy, which consist in choosing an optimal sub-set from a large set of candidate points. The selecting procedure is inspired by the notion of \textit{Fekete points} in polynomial interpolation community, and the method only involves a simple QR type decomposition for the design matrix. It is shown that the selected quasi-optimal points distributed asymptotically according to the equilibrium measure, and we thus solve a Christoffel-preconditioned least-square approach. Numerical examples show that our method demonstrates much improved stability property than the traditional sampling strategies.
\end{abstract}

%\begin{keywords}
%Uncertainty quantification, least squares approximations, Fekete points, QR decomposition
%\end{keywords}

\pagestyle{myheadings}
\thispagestyle{plain}

\section{Introduction}\label{sec:methodology}

The construction of polynomial surrogates that emulate a system response with respect to input parameters is a widely-used tool in computational science. A concrete example is provided by problems in parametric uncertainty quantification (UQ), where this approach is frequently called generalized Polynomial chaos (gPC) \cite{Ghanem_book_1991,Xiu_2002wiener}. The standard approach is to consider a scalar function $f(y)$ depending on inputs/parameters $y \in \Gamma \subset \mathbb{R}^d$, and to approximate this function with a multivariate polynomial expansion. The parameter $y$ is usually interpreted as a random parameter $Y$, and the basis chosen to perform the expansion is one whose elements $\psi_j$ are orthonormal under the density of $Y$,
\begin{align}\label{eq:f-approx}
  f(y) \approx \sum_{j=1}^N \widehat{f}_j \psi_j(y)
\end{align}
%The basis chosen  whose orthogonality is defined by the probability measure of the input variable $\mathbf{Y}$.
Accurately estimating the coefficients $\widehat{f}_j$ of the expansion is important since these coefficients can be easily manipulated to infer revealing properties, such as statistical moments or parametric sensitivities.  Many numerical techniques on how to obtain the polynomial coefficients in UQ problems have been developed in recent years. While early development often focused on ``intrusive" methods, such as stochastic Galerkin, much recent effort has concentrated on non-intrusive-type collocation methods \cite{Xiu_2005highorder,Eldred_2009advances,narayan_stochastic_2015}. In the collocation framework, one seeks to compute the expansion coefficients via point-evaluations of $f$, and thus constructing a ``good" configuration of samples in $\Gamma$ has become an active area of research. Popular methods include sparse grids \cite{Zabaras_2007sparsegrid,Fabio,Xiu_2005highorder,narayan_adaptive_2014} and polynomial interpolation \cite{Narayan_X_SISC_2012,Cohen_interpolation}. The particular numerical method one uses to compute expansion coefficients often influences the particular sampling strategy, as evidenced by reserach on sparse approximations using $\ell_1$-minimization \cite{Doostan_2011nonada,Xu_2014Weilpoint,Yan_GX_IJUQ_2012,JNZ_2016}.

In this paper we focus on computing coefficients via the discrete least-squares approach using point-evaluations of $f$. While it is relatively easy to compute the coefficients via standard linear algebraic operations, least-squares approaches have known stability issues. For instance, when using Newton-Cotes quadrature abscissae (equidistant point sets) it is highly unstable even for an infinitely smooth noiseless function unless significant oversampling is performed. Several sampling strategies has been proposed in recent years \cite{Cohen_2013staac,Migliorati_2014analysisL2,Zhou_2015wdls,Hampton_2015coherence,Tang_2014discretels,Narayan_2014Christoffel} to improve the stability for least-squares. These methods use both randomized and deterministic sampling methods. %Strategies in the above mentioned paper includes random sampling, deterministic sampling and equilibrium sampling with weighted least-squares.

In this paper we propose and analyze a greedily-computed deterministic sample set for discrete least-squares, where the objective in the greedy process is a weighted determinant. A discrete-least squares system for computing coefficients in \eqref{eq:f-approx} from $M$ point evaluations of $f$ utilizes a Vandermonde-like matrix $\bs{V} \in \R^{M \times N}$, and our procedure greedily forms a set of points $A$ via:
\begin{align*}
  A &\gets A \cup y_{n+1}, & y_{n+1} = \argmax_{y \in \Gamma} \det| \bs{W} \bs{V} \bs{V}^T \bs{W} |,
\end{align*}
where $\bs{W}$ is a diagonal matrix containing the weights whose entries are the $L^2$ Christoffel function associated with it. The precise procedure is given in \eqref{eq:Q-greedy-F}. Without the weights, this procedure is essentially the method of computing Approximate Fekete points \cite{sommariva_computing_2009,Bos_2011Geometric}, and so computationally this is easily implemented with pivoted linear algebraic routines. Base on these connections, we call our procedure Christoffel-weighted approximate Fekete points (CFP).

The introduction of the weights introduces mathematically nontrivial results, and our procedure results in the following theoretical and practical advantages:
\begin{itemize}
  \item Approximate Fekete points are restricted to compact sets $\Gamma$. The CFP formulation is impartial to (non-)compactness of $\Gamma$.
  \item Under certain assumptions, the greedy CFP procedure generates a sequence that coincides with the result of simultaneous/global optimization. (That is, maximizing the weighted determinant by varying $y_1, y_2, \ldots$ simultaneously.) The required assumptions are in practice difficult to verify in multidimensional settings, but that our greedy procedure can in principle produce the same result as simultaneous optimization is a strong advantage of the method.
  \item For general distributions of $Y$ in one dimension, the requisite assumptions of the previous bullet point are essentially always satisfied. Thus, in one dimension our greedy design almost  always\footnote{The initial sample for the greedy method can be any point on $\R$ not coinciding with $N-1$ isolated points; see Thereom \ref{thm:1d-optimal} for details.} produces an optimal mesh. In this particular one-dimensional setting, the CFP algorithm produces abscissae for Gauss quadrature rules.
  \item Like the methodology for approximate Fekete points, the computational procedure for CFP requires essentially only pivoted linear algebra routines, in particular the $QR$ factorization.
  \item Our numerical results show that CFP produces empirically superior results when compared with the deterministic sampling strategy given by approximate Fekete points, and when compared against randomized Monte Carlo sampling methods. This is true for all our test cases, both in low-dimensional settings ($d=2$) and in relatively high-dimensional settings ($d=25$).
\end{itemize}
%However, selecting the optimal samples in the least-square approach is still a challenging work, as one has to search the points in the whole parametric domain (that contains infinite many points). This work consider a more practical and solvable problem:
%\begin{itemize}
%\item \textit{Suppose we have a large set of candidate points, then how to choose an optimal sub-set of samples for the least-squares polynomial approach?}
%\end{itemize}
%To this end, we borrow the notion of \textit{Fekete points} in the polynomial interpolation community.  The selection strategy we proposed only involves a simple QR type decomposition for the design matrix. We shown that the selected quasi-optimal points distributed asymptotically according to the equilibrium measure, and we thus solve a Christoffel-preconditioned least-squares approach. Comparisons with other sampling strategies are presented, and it is noticed that our method demonstrates much improved stability property.

%{\color{red}Comment Dongbin's work \cite{Shin_Xiu} here?}
The paper is organized as following. In section 2, we introduce notation and discuss least-squares problems. Our approach for grid design is introduced in Section 3, along with a description of theoretical properties (with proofs provided in the Appendix). Section 4 describes some details of the algorithm, and section 5 numerically investigates several examples.

\subsection{Historical discussion}
%The least-squares approach is relatively easy to implement. However, care should be paid to the sampling strategy to address the stability issue. Recently, theoretical analysis have been focus on determining the stability conditions for least squares for various sampling strategies.
This section describes some previous theoretical results on sampling for polynomial least-squares. The authors in \cite{Cohen_2013staac} provide foundational theoretical analysis for unweighted discrete least-squares using Monte Carlo sampling.  For an $N$-term expansion in tensor-product Legendre polynomials, iid sampling from the uniform distribution requires $M\sim N^2$ points to guarantee the stability with high probability. With expansions in tensor-product Chebyshev polynomials, the condition can be reduced to $M\sim N^{\frac{\ln 3}{\ln 2}}$ \cite{Chkifa_2015discretels}. More general results can be derived using the inequalities in \cite{Migliorati_2013Aq,migliorati_multivariate_2015}.

Weighted least-squares approaches with Monte Carlo samples have also been investigated. Analysis for general weighted procedures is given in \cite{Narayan_2014Christoffel}, where the authors also observe that sampling from the weighted pluripotential equilibrium measure provides optimal stability and convergence estimates for approximations with asymptotically large polynomial degree. The authors in \cite{Hampton_2015coherence} propose an inexact sampling method for optimal sampling in the non-asymptotic case. The results in \cite{cohen_optimal_2016} suggest an exact sampling method and show optimal convergence estimates in the non-asymptotic case. The work in \cite{narayan_computation_2017} provides efficient computational methods for exact sampling in the non-asymptotic case.

%The authors in with equilibrium sampling strategy with Christoffel precondition technique is proposed in \cite{Narayan_2014Christoffel}, and it is noticed that the approach it optimal in the asymptotic regime and a linear dependence $M\sim CN$ is enough to guarantee the stability. Furthermore, the framework in \cite{Narayan_2014Christoffel} applies to general distributions. The authors in \cite{cohen_optimal_2016} propose a non-asymptotic procedure for optimal sampling.

The focus of this paper is on determinstic sampling schemes; such methods have also been investigated \cite{Migliorati_2014analysisL2,Xu_2014Weilpoint,calvi_uniform_2008}. We remark again that polynomials grids constructed via Fekete or Leja methods are closely connected to our procedure \cite{bos_calculation_2008,Bos_2010MLFeketeLeja,Bos_2011Geometric}. Another closely related approach is provided in \cite{Shin_Xiu}, wherein the authors optimize a determinant-like objective.
%In particular, a novel deterministic sample candidates called the Weil points was investigated in \cite{Xu_2014Weilpoint}. It shows that the Weil points distribute asymptotically according to the Chebyshev measure. To recover the Chebyshev expansions, the stability can be guaranteed provided that $M\sim N^2$ by using the Weil points. Notice that in this framework, one can show the stability with probability 1.

\section{Problem formulation}
This section focuses on introduction of notation and some background material. More details on discrete least-squares problems can be found in \cite[Chaps. 10-11]{Smith_2013UQ}.

%In this section, we briefly review the procedure of the least-squares polynomial approximations. For more details of discrete least-squares approximations, we refer the reader to Chapters 10 and 11 of \cite{Smith_2013UQ}.

%\subsection{Notation}

Let $y = (y_1 , ... , y_d)^\top$ be a $d$-dimensional vector whose components take values in $\Gamma_i \subset \R$. In parametric uncertainty quantification problems, each $y_i$ corresponds to a random variable input into a system, and the goal is understand how the system depends on these inputs. This is frequently done via a linear expansion in a basis of polynomials that is orthogonal with respect to the $L^2$ norm weighted by the probability density function, i.e., a polynomial Chaos expansion \cite{Ghanem_book_1991,Xiu_2002wiener}.

\subsection{Notation}\label{sec:notation}
We assume that $\Gamma$ is tensorial and $\rho(y)$ is a tensor-product probability density on $\Gamma \subset \R^d$. In parametric UQ problems, this is equivalent to assuming that the components of a random variable $Y = (Y_1, \ldots, Y_d)$ are mutually independent, and that $Y_i$ has marginal probability density $\rho_i: \Gamma_i \rightarrow [0, \infty)$. Thus $\Gamma:= \prod_{i=1}^d\Gamma_i \subset \mathbb{R}^d ,$ and $\rho(y)= \prod_{i=1}^d \rho_i(y_i): \Gamma\rightarrow \mathbb{R}^+$. All our results hold for variables $Y_i$ that are discrete, or mixtures of discrete and continuous random variables but for simplicity we assume throughout that they are continuous random variables with densities.

For each $i = 1, \ldots, d$, we can define orthogonality in terms of the $\rho_i$-weighted $L^2$ norm on $\Gamma_i$. Assuming that $0 < \E\,|Y_i|^{k} < \infty$ for all $k \in \N_0$, then $Y_i$ has finite moments of all orders. This ensure existence of a family of orthogonal polynomials. We let $\varphi^i_n$ denote the degree-$n$ polynomial from the orthogonal family associated to the weight function $\rho_i$. Therefore,
\begin{align*}
  \int_{\Gamma_i} \varphi^i_k(y_i) \varphi^i_l(y_i) \rho_i(y_i) dy_i &= \delta_{k,l}, & i=1,\ldots,d,\; &\; k, l =0, 1, \ldots,
\end{align*}
where $\delta_{k,l}$ is the Kronecker delta function.

Since $\rho$ is a tensor-product weight on a tensorial domain $\Gamma$, then multivariate polynomials orthogonal under $\rho$ can be formed via tensorization: For any multi-index $\alpha \in \N_0^d$, the polynomials defined as
\begin{align*}
  \psi_{\alpha}(y) &\coloneqq \prod_{j=1}^d \varphi^{j}_{\alpha_j}\left(y_{j}\right), & \alpha \in \N_0^d
\end{align*}
satisfy
\begin{align*}
  \left\langle \psi_\alpha, \psi_\beta \right\rangle \coloneqq \int_{\Gamma} \psi_\alpha(y) \psi_\beta(y) \rho(y) dy &= \delta_{\alpha,\beta}, & \alpha, \beta &\in \N_0^d
\end{align*}
With $\alpha = \left( \alpha_1, ..., \alpha_d\right) \in \mathbb{N}_0^d$ a multi-index, then $|\alpha| = \sum_{j=1}^d \alpha_j$. Various polynomial subspaces can be defined by identifying an appropriate collection of multi-indices. For example, the total degree and hyperbolic cross index sets of order $k$ are, respectively,
\begin{align*}
  \Lambda_k^{\textrm{TD}} &\coloneqq \left\{ \alpha \;\; \Big| \;\; |\alpha| \leq k \right\}, &
  \Lambda_k^{\textrm{HC}} &\coloneqq \left\{ \alpha \;\; \Big| \;\; \prod_{j=1}^d (\alpha_j+1) \leq k+1 \right\}
\end{align*}
It will occasionally be convenient to place an ordering on indices in a finite index set. If $\Lambda \subset \N_0^d$ has size $N$, then we will assume an implicit one-to-one correspondence between the sets
\begin{align}\label{eq:index-ordering}
  \left\{ \alpha \;\; | \;\; \alpha \in \Lambda \right\} \longleftrightarrow \left\{ 1, \ldots, N \right\}.
\end{align}
We let $\alpha(j)$ denote the index corresponding to $j$ from the above map, so that $\alpha(1), \ldots, \alpha(N)$ is an ordering of the elements of $\Lambda$. In our context, the particular choice of correspondence defined above is irrelevant.

For any index set $\Lambda$, we define the associated space of polynomials as
\begin{align}\label{eq:P-def}
  P = P(\Lambda) = \mathrm{span} \left\{ \psi_\alpha \;\; | \;\; \alpha \in \Lambda \right\}
\end{align}
%For a given multi-index, $\alpha$, we define the set of multi-indices, $\mathcal{I}_\alpha=\{
%\mathbf{i} : \mathbf{i} \in \mathbb{N}^d; 1 \leq i_k \leq \alpha_k; k = 1,...,d\}$. This
%index set can be used to reference components for tensor product extensions. The most popular index sets are the total order
%index sets given by $ \sum_j i_j \leq m $. We use $\Lambda $ to denote a total order index sets with cardinality $\sharp\Lambda=\left(^{m+d}_{\,\,\,\,d}\right):=N$. One can also consider the tensor order indices
%(analogous to tensor products) and hyperbolic cross index sets \cite{}.  In this paper, we just focuses on the
%total degree index sets.
%Then, a multivariate orthogonal polynomial space with a total order index set $\Lambda$ can be defined as
%\begin{align*}
% P_\Lambda := \text{span}\{\psi_{\alpha}(\mathbf{Y})| \alpha \in \Lambda\}.
%\end{align*}
Via the map \eqref{eq:index-ordering}, the functions $\psi_j \coloneqq \psi_{\alpha(j)}$, $j=1, \ldots, N$ form an orthonormal basis for $P(\Lambda)$.
%We can regularly alternate between integer and multi-index notation for the associated polynomial ansatz:
%\begin{align*}
%  \big\{\psi_{\alpha}(\mathbf{Y})\big\}_{\alpha \in \Lambda}\,\, \Leftrightarrow \,\,\big\{\psi_n(\mathbf{Y})\big\}_{n=1}^{N}.
%\end{align*}
We now define a set of \textit{weighted} polynomials. With
\begin{align*}
  K_\Lambda(y) = \sum_{\alpha \in \Lambda} \psi_\alpha^2(y),
\end{align*}
then $Q(\Lambda)$ is the following space of weighted polynomials:
\begin{align}\label{eq:Q-def}
  Q = Q(\Lambda) = \mathrm{span} \left\{ \frac{\psi_\alpha}{\sqrt{K_\Lambda}} \;\; \big| \;\; \alpha \in \Lambda \right\}.
\end{align}
Given any points $y_1, y_2, \ldots \in \Gamma$, we will use the notation
\begin{align*}
  A_m &= \left\{ y_1, \ldots, y_m \right\}, & m &\geq1
\end{align*}
to denote a size-$m$ set of points. An $m \times N$ Vandermonde-like matrix for $A_m$ on $P(\Lambda)$ using an orthonormal basis is given by
\begin{align}\label{eq:V-def}
  \bs{V}\left(A_m, P\right) &\in \R^{m \times N}, & (V)_{j,k} &\coloneqq \psi_{\alpha(k)}(y_j),
\end{align}
for $1 \leq j \leq m$, and $1 \leq k \leq N$, with $\alpha(k)$ as defined in \eqref{eq:index-ordering}. One advantage of our using an orthonormal basis is that the Vandermonde-like matrix using \textit{any} other orthonormal basis for $P(\Lambda)$ equals $\bs{V}\left(A_m, P(\Lambda)\right) \bs{U}$ for some orthogonal matrix $\bs{U}$.

We define $\bs{V}\left(A_m, Q\right)$ similarly, having elements $\frac{\psi_{\alpha(k)(y_j)}}{\sqrt{K_\Lambda(y_j)}}$.

\subsection{Discrete least-squares problems}
Given a function $f: \Gamma \rightarrow \R$ and multi-index set $\Lambda$, our main goal is to construct a polynomial approximation $f_N$ from $P(\Lambda)$:
\begin{align}\label{eq:f-exp}
  f(y) \approx f_N(y) \coloneqq \sum_{n = 1}^N \widehat{f}_n \psi_n(y)
\end{align}
The $L^2_{\rho}(\Gamma)$-best approximation from $P(\Lambda)$ is the polynomial
\begin{align*}
  f_N^\ast(y) = \sum_{n=1}^N \left\langle f, \psi_n \right\rangle \psi_n(y).
\end{align*}
The coefficients of this polynomial clearly require significant information about the function $f$ via the inner products $\left\langle f, \psi_n \right\rangle$. In practice such information cannot be computed directly and instead only point evaluations $f(y)$ at a discrete number $y$ values are possible to obtain. If $A_M = \left\{ y_1, \ldots y_M\right\} \subset \Gamma$ is some selection of $M$ points, then one possible method to compute $f_N$ is to compute the least-squares residual minimizer, which is a quadratic optimization problem whose solution is linear in the data $f(y_m)$:
\begin{align}\label{LS}
  \left\{\widehat{f}_n\right\}_{n=1}^N \;\; = \;\; \argmin_{f_n} \sum_{m=1}^M \left[ f(y_m) - \sum_{n=1}^N \widehat{f}_n \psi_n(y_m) \right]^2
\end{align}
This problem can be written algebraically: Let $\mathbf{V}(A_M,P)$ be the $M\times N$ Vandermonde-like matrix defined in \eqref{eq:V-def}. We collect the unknown coefficients $\widehat{f}_n$ into the vector $\widehat{\mathbf{f}}$, and  collect the function evaluations $f(y_m)$ into
the vector $\mathbf{f}\in \mathbb{R}^M$. The least-squares approach (\ref{LS}) is equivalent to
\begin{equation}\label{Lsequation}
  \mathbf{\widehat{f}}=\arg\min_{\mathbf{v}\in \mathbb{R}^N}\|\mathbf{V}(A_M,P)\mathbf{v}-\mathbf{f}\|_2,
\end{equation}
which is simple algebraic least-squares problem.

Recent research has shown that the unweighted least-squares formulation above is frequently inferior to a particular weighted approach \cite{Hampton_2015coherence,Narayan_2014Christoffel,cohen_optimal_2016}. This approach uses weights given by $1/K_\Lambda$. The algebraic formulation of the weighted approach solves
\begin{equation}\label{eq:weighted-Lsequation}
  \mathbf{\widehat{f}}=\arg\min_{\mathbf{v}\in \mathbb{R}^N}\|\mathbf{V}(A_M,Q)\mathbf{v}-\bs{W}\mathbf{f}\|_2,
\end{equation}
where $\bs{W}$ is a diagonal matrix with entries $(W)_{m,m} = 1/\sqrt{K_\Lambda(y_m)}$, $1 \leq m \leq M$. This paper focuses on solving \eqref{eq:weighted-Lsequation}, where we use a deterministic sampling approach to compute $A_M$.

\section{A Quasi-optimal sampling strategy}

Our sampling strategy relies on the notion of Fekete points for polynomial interpolation, to this end, we first review some basic definitions for the Fekete points. Throughout this section $\Lambda$ is an arbitrary but fixed finite multi-index set with size $N$, and we use the abbreviations $P = P(\Lambda)$ and $Q = Q(\Lambda)$ as defined in \eqref{eq:P-def} and \eqref{eq:Q-def}, respectively.

The CFP method we propose in this paper is provided by the greedy optimization \eqref{eq:Q-greedy-F}, but the first three subsections below provide motivating discussion for this optimization.

\subsection{Determinants and interpolation}
One set of good points, in theory, for polynomial interpolation is the Fekete points. %Although the polynomial interpolation is a classical subject, we need to briefly outline
%the main features of (multivariate) polynomial interpolation to make the notations precise.
We give a brief discussion of this below, but for more details along this line we refer to \cite{Bos_2008calculation, Bos_2010MLFeketeLeja} and references therein.

Assume in this section that $\Gamma \subset \mathbb{R}^d$ is a compact set with nonempty interior. %, and let $P_n = P\left(\Lambda_n^{\mathrm{TD}}\right)$ denote the space of polynomials of degree at most $n$, with $N = \dim P_n$.
Given a set of $N$ distinct points $\mathcal{A}_N=\{y_i\}_{i=1}^{N}\subset \Gamma$ and a function $f: \Gamma\to \mathbb{R}$, the polynomial interpolation problem is to find a $p\in P_\Lambda(\Gamma)$ such that
\begin{equation}\label{interpolationcondition}
p(y_i)=f(y_i),\quad \forall y_i\in \mathcal{A}_N.
\end{equation}
We assume that this problem is unisolvent; this is true unless the $y_i$ have a pathological configuation in $\Gamma$\footnote{For example, if we choose $y_i$ as realizations of a continuous random variable that is uniform on $\Gamma$, then the interpolation problem is unisolvent with probability 1.}. With $\{\psi_1,\psi_2,...\psi_{N}\}$ any ordered basis for $P(\Lambda)$, then there are unique coefficients $c_j$ satisfying a linear system that determines $p$:
%\begin{equation*}
%p = \sum\limits_{j=1}^{N}c_j\psi_j
%\end{equation*}
%These coefficients can be determined by solving the algebraic form of the interpolation problem associated to (\ref{interpolationcondition}):
\begin{align}\label{inperpolationmatrix}
  p &= \sum\limits_{j=1}^{N}c_j\psi_j & \mathbf{V} (A_N; P_\Lambda)\mathbf{c}&=\mathbf{f},
\end{align}
%where $\mathbf{c}$ is the unknown coefficient vector, and $\mathbf{f}\in \mathbb{R}^{N}$ is the evaluation vector of  $f$ on the interpolation points set $\mathcal{A}_N$.
%When $\bs{V}$ is a rectangular matrix we define $\det \bs{V}$ in the standard way; when it is rectangular we define its determinant modulus as follows.
\begin{definition}\label{def:det}
  Let $y_j$, $j=1, \ldots, N$ denote any set of points in $\Gamma$, so that $A_m = \left\{ y_j \right\}_{j=1}^m$ for $m \leq N$ is well-defined. Define the determinant modulus of the rectangular matrix $\bs{V} = \bs{V}(A_m, P)$ as
  \begin{align*}
    \left| \det\bs{V}\right| &= \sqrt{ \left| \det{\left(\bs{V} \bs{V}^T\right)} \right|}, & 1 \leq m &\leq N
  \end{align*}
\end{definition}
The definition above coincides with the standard square matrix determinant (modulus) when $m = N$. Some additional notation is the operation of appending a point to a given set $A_m$, and replacing the $j$th element of $A_N$, respectively:
\begin{align*}
  A_{m \cup y} &= \left\{ y_1, \ldots, y_m, y \right\}, & y &\in \Gamma. \\
  A_{N\backslash j(y)} &\coloneqq \left\{ y_1, \ldots, y_{j-1}, y, y_{j+1}, \ldots, y_N \right\}, & 1 \leq j \leq N, \hskip 5pt y &\in \Gamma.
\end{align*}
If the polynomial interpolation problem on $P_\Lambda$ is on $\mathcal{A}_N$, then the cardinal Lagrange interpolation polynomials are given by
\begin{align*}
  \ell_j(y) &= \frac{ \det \bs{V}\left( A_{N\backslash j(y)}, P_\Lambda \right)}{ \det \bs{V}\left(A_N, P_\Lambda \right)}, & \ell_j(y_k) = \delta_{j,k},
\end{align*}
with $\delta_{j,k}$ the Kronecker delta. This allows us to explicitly construct the unique element $p \in P(\Lambda)$ that interpolaes at $A_N$ an arbitrary $f$ continuous on $A_N$:
\begin{align*}
  I_N f \coloneqq \sum_{n=1}^N f(y_n) \ell_n(y).
\end{align*}
The output of this operator is an element of $P(\Lambda)$, which we can view as a subspace of $C(\Gamma)$, continuous functions over $\Gamma$. This results in a popular notion of stability, the Lebesgue constant,
\begin{align*}
  \left\| I_N \right\|_{C(\Gamma) \rightarrow C(\Gamma)} &= \sup_{y \in \Gamma} \sum_{n=1}^N \left|\ell_n(y)\right|. \\
  %\left\| I_N \right\|_{\ell^2(A_N) \rightarrow L^2(\Gamma)} &= \frac{1}{\sigma_{\mathrm{min}}\left( \bs{V}\left(A_N, P(\Lambda) \right)\right)}.
\end{align*}
This quantity does not depend on which basis for $P(\Lambda)$ is chosen to compute the cardinal Lagrange interpolants. Finally, the conditioning of the problem of computing $\bs{c}$ from an arbitrary $\bs{f}$ is measured by
\begin{align*}
  \kappa\left( \bs{V}\left(A_N, P \right)\right) = \frac{\sigma_{1}\left(\bs{V}\right)}{\sigma_{N}\left(\bs{V}\right)},
\end{align*}
where $\sigma_j$, $j=1, \ldots, N$, are the singular values of $\bs{V}$ in decreasing order.

\subsection{Near-optimal stability and conditioning}
Consider the square systems case, $M = N$, with interpolation on $P = P(\Lambda)$. A set of points that maximizes the determinant of the Vandermonde-like matrix is called a set of Fekete points:
\begin{align}\label{eq:fekete-points}
  A_N^F(P(\Lambda)) \coloneqq \argmax_{A_N = \left\{y_1, \ldots, y_N\right\} \in \Gamma^N} \left| \det \bs{V}\left(A_N, P \right) \right|.
\end{align}
Classically, Fekete points on general manifolds are point configurations that minimize a Reisz energy. On a compact interval in $\R$, a specialization of Riesz energy coincides with the determinant of the Vandermonde matrix. Thus, in the one-dimensional setting a set of Fekete points is determined by maximizing the determinant of the Vandermonde matrix.

The utility of Fekete points for polynomial approximation is that they provide at-most-linear growth of the Lebesgue constant:
\begin{align*}
  \left\| I_N \right\|_{C(\Gamma) \rightarrow C(\Gamma)} = \sup_{y \in \Gamma} \sum_{n=1}^N \left| \frac{ \det \bs{V}\left( A^F_{N\backslash n(y)}, P\right)}{ \det \bs{V}\left(A^F_N, P\right)} \right| \leq \sup_{y \in \Gamma} \sum_{n=1}^N 1 = N
\end{align*}
In practice, logarithmic growth is observed. Thus, the computation of determinant-maximizing sample points \eqref{eq:fekete-points} is of great interest.

We can simiarly define an optimization problem that seeks a point configuration with minimal condition number:
\begin{align}\label{eq:optimal-conditioning}
  A_N^C(P) \coloneqq \argmin_{A_N = \left\{y_1, \ldots, y_N\right\} \in \Gamma^N} \kappa \left(\bs{V}\left(A_N, P \right)\right).
\end{align}
Note that $A_N^F$ and $A_N^C$ are different sets in general.

\subsection{Greedy designs}
The optimization problems \eqref{eq:fekete-points} and \eqref{eq:optimal-conditioning} are not computationally feasible in general, and so one frequently results to greedy algorithms. Greedy versions of these algorithms are straightforward to devise:
\begin{subequations}\label{eq:greedy-P}
\begin{align}
  A_N^{F\ast}(P) &= \left\{ y_1^{F\ast}, \ldots, y_N^{F\ast} \right\}, & y_{n+1}^{F\ast} &= \argmax_{y \in \Gamma} \left| \det\bs{V}\left(A_{n\cup y}^{F\ast}, P\right) \right| \\
  A_N^{C\ast}(P) &= \left\{ y_1^{C\ast}, \ldots, y_N^{C\ast} \right\}, & y_{n+1}^{C\ast} &= \argmin_{y \in \Gamma}  \kappa\left(\bs{V}\left(A_{n\cup y}^{C\ast}, P\right) \right)
\end{align}
\end{subequations}
These greedy versions are still difficult, but are more feasible since they involve only repeated optimization over $\Gamma$ (instead of optimization over $\Gamma^N$). The determinant-maximizing objective is easier to compute compared to the condition number objective. In pratice, one often replaces exact maximization over $\Gamma$ with maximization over a discrete set. Finally, there is ambiguity at each iteration if multiple locations $y$ maximize objectives, and there is freedom in choosing the starting point $y_1$ in each case.

There are two major difficulties with all of our previous discussions: First, we now have four potential sets, $A^F$, $A^C$, $A^{F\ast}$, and $A^{C\ast}$ that we would like to compute. It seems unclear, for example, whether $A^{F\ast}$ or $A^{C\ast}$ is the better option. Our second difficulty, is that none of these sets is well-defined if $\Gamma$ is not compact.

\subsection{Weighted greedy designs}
We can partially resolve the difficulties identified at the end of the previous section by considering weighted polynomials. By doing this, we show under some assumptions that greedy designs can produce the same result as the much more burdensome simultaenous optimization designs. In addition, we show in one dimension for any $\rho$ that this almost always happens.

%When $\Gamma$ is one-dimensional, we can resolve both of the difficulties identified at the end of the previous section by considering weighted polynomials.
We reformulate all four problems, both the optimal versions \eqref{eq:fekete-points} and \eqref{eq:optimal-conditioning}, as well as their greedy versions \eqref{eq:greedy-P}. Let $\Gamma \subset \R^d$, $d \geq 1$, and now assume only that $\Gamma$ has an interior containing any open set with positive Lebesgue measure. Instead of working on the polynomial space $P = P(\Lambda)$, we'll use the weighed space $Q = Q(\Lambda)$ defined in \eqref{eq:Q-def}. The point configurations that maximize the determinant, and minimize the condition number, respectively, are defined as
\begin{subequations}\label{eq:Q-optimal}
  \begin{align}\label{eq:Q-optimal-F}
    A_N^F(Q) \coloneqq \argmax_{A_N = \left\{y_1, \ldots, y_N\right\} \in \Gamma^N} \left| \det \bs{V}\left(A_N, Q \right) \right|, \\\label{eq:Q-optimal-C}
A_N^C(Q) \coloneqq \argmin_{A_N = \left\{y_1, \ldots, y_N\right\} \in \Gamma^N} \kappa\left( \bs{V}\left(A_N, Q \right) \right),
\end{align}
\end{subequations}
and the greedy versions are, for $n \geq 1$,
\begin{subequations}\label{eq:Q-greedy}
  \begin{align}\label{eq:Q-greedy-F}
    A_N^{F\ast}(Q) &= \left\{ y_1^{F\ast}, \ldots, y_N^{F\ast} \right\}, & y_{n+1}^{F\ast} &= \argmax_{y \in \Gamma} \left| \det\bs{V}\left(A_{n\cup y}^{F\ast}, Q\right) \right| \\\label{eq:Q-greedy-C}
  A_N^{C\ast}(Q) &= \left\{ y_1^{C\ast}, \ldots, y_N^{C\ast} \right\}, & y_{n+1}^{C\ast} &= \argmin_{y \in \Gamma}  \kappa\left(\bs{V}\left(A_{n\cup y}^{C\ast}, Q\right) \right)
\end{align}
\end{subequations}
In the above, the starting values $y_1^{F\ast}$ and $y_1^{C\ast}$ can take arbitrary values in $\R$, but thus choice affects the final result of the greedy pursuit. Furthermore, at each $n$ there may be multiple points $y^{F\ast}_{n+1}$ that extremize the objective. We assume that any one of these extremizers are chosen for the procedure, and refer below to this potential non-uniqueness of the sequence as a \textit{branch} of the iterative optimization.

The greedy algorithms working on $Q$ are just as computationally feasible as those working on $P$. However, the advantage of this particular weighted approach is that if an optimal solution exists then all of the four approaches \eqref{eq:Q-optimal} and \eqref{eq:Q-greedy} give the optimal solution.
\begin{theorem}\label{thm:nd-optimal}
  Let $\rho: \Gamma \rightarrow [0, \infty)$ be a probability density on $\R^d$. Let $\Lambda$ be an arbitrary multi-index set of size $N$ defining $Q$. A point configuration $A_N$ satisfies
  \begin{subequations}\label{eq:nd-optimal}
  \begin{align}\label{eq:nd-optimal-det}
    \left| \det \bs{V}\left( A_N, Q \right) \right| = 1,
  \end{align}
  if and only if
  \begin{align}\label{eq:nd-optimal-cond}
    \kappa\left( \bs{V}\left( A_N, Q \right) \right) = 1.
  \end{align}
  \end{subequations}
  Thus, solutions to \eqref{eq:Q-optimal} attaining optimal objective values coincide. If $A_N$ is a set that satisfies either (hence both) of the optimal objective values above, then:
  \begin{itemize}
    \item If $y_1^{F\ast} \in A_N$ then the iteration \eqref{eq:Q-greedy-F} has a branch such that $A_N^{F\ast} = A_N$.
    \item If $y_1^{C\ast} \in A_N$ then the iteration \eqref{eq:Q-greedy-C} has a branch such that $A_N^{C\ast} = A_N$.
  \end{itemize}
\end{theorem}
\begin{proof}
  See Appendix \ref{thm:nd-optimal-proof}.
\end{proof}

The strength of this result is twofold: first, it suggests that we may either optimize the determinant or the condition number and obtain equivalent answers. Second, it shows that greedy optimization recovers the global optimum. These conclusions give us great flexibility in computational procedures since we may propose a method for a greedy determinant maximization and this plausibly gives results comparable to global minimization of the condition number.

The unfortunate caveat in the result above is that we require existence of a point configuration with optimal condition number and determinant. It is initially unclear whether or not this is a reasonable assumption. However, nontrivial multidimensional examples for when this condition is satisfied exist \cite{jakeman2017}.

A somewhat surpising positive result is that in one dimension $(d=1)$, infinitely many point configuations with optimal condition number exist, and the union of all these optimal sets covers every real number, except for $N-1$ isolated points.
\begin{lemma}\label{lemma:1d}
  Let $\rho(y)$ be any probability density function on $\Gamma = \R$, and let $\Lambda = \left\{ 0, \ldots, N-1 \right\}$ for any $N \geq 1$. Recall that $\varphi_N(\cdot)$ denotes the degree-$N$ orthonormal polynomial with respect to the $\rho$-weighted $L^2$ inner product on $\Gamma$. We use $\varphi_{N-1}^{-1}(0)$ to denote the zero set of the polynomial $\varphi_{N-1}(y)$, which is always a set of $N-1$ distinct points in $\R$. Then:
  \begin{enumerate}
    \item For any $y \not \in \varphi_{N-1}^{-1}(0)$, there is a set $A_N = A_N(y)$ that satisfies \eqref{eq:nd-optimal}.
    \item The set $A_N(y)$ defined above is unique as a function of $y$.
    \item The set $A_N(y)$ is given by
      \begin{align*}
        A_N(y) = r_N^{-1}\left( r_N(y) \right),
      \end{align*}
      where $r_N$ is the meromorphic function
      \begin{align*}
        r_N(y) = \frac{\varphi_N(y)}{\varphi_{N-1}(y)},
      \end{align*}
      and $r_N^{-1}$ denotes its set-valued functional inverse.
    \item The set $A_N(y)$ contains the abscissae for an $N$-point positive quadrature rule exact for polynomials up to degree $2 N - 2$,
      \begin{align*}
        \int_\Gamma p(z) \rho(z) d z &= \sum_{z \in A_N(y)} \frac{1}{K_\Lambda(z)} p(z), & \deg p &\leq 2 N - 2.
      \end{align*}
    \item If $y \in \varphi_{N}^{-1}(0)$, then $A_N(y)$ are the abscissae of the $N$-point Gauss quadrature rule.
  \end{enumerate}
\end{lemma}
\begin{proof}
  See Appendix \ref{lemma:1d-proof}.
\end{proof}

Note that in the lemma above we must consider $\rho$ a density on $\Gamma = \R$, even if its support lies on a compact set. This result shows that in one dimension (for \textit{any} $\rho$) there are many optimal point configurations.

\begin{theorem}\label{thm:1d-optimal}
  Let $\rho$ be any probability density function on $\Gamma = \R$, and let $\Lambda = \left\{ 0, \ldots, N-1 \right\}$ for any $N \geq 1$. Let $y \in \varphi_{N-1}^{-1}(0)$ be arbitrary but fixed, and let $A_N(y)$ denote the unique set defined in Lemma \ref{lemma:1d}. Set $A_N^F(Q) = A_N^C(Q) = A_N(y)$, which satisfy the maximization problems \eqref{eq:Q-optimal}. Set the initial values for the greedy iterations \eqref{eq:Q-greedy} as
  \begin{align*}
    y_1^{F\ast} = y_1^{C\ast} = y.
  \end{align*}
  Then
  \begin{align*}
    A_N^{C}(Q) = A_N^{C\ast}(Q) = A_N^{F\ast}(Q) = A_N^{F}(Q),
  \end{align*}
  and
  \begin{align*}
    \left| \det \bs{V}\left( A_N^{F}, Q \right) \right| = \left| \det \bs{V}\left( A_N^{F\ast}, Q \right) \right| = 1 = \kappa \left( \bs{V}\left( A_N^{C\ast}, Q \right) \right) = \kappa\left( \bs{V}\left( A_N^{C}, Q\right) \right)
  \end{align*}
\end{theorem}
\begin{proof}
  See appendix \ref{thm:1d-optimal-proof}.
\end{proof}

We emphasize that the above result holds for \textit{any} univariate density $\rho$, even those with non-compact support. This result completely characterizes the greedy scheme's behavior in one dimension, and shows that it achieves an optimal condition number for almost any starting point $y$.

One final observation is that the set $A_N^{F\ast}$ produced by the one-dimensional greedy iteration produce optimal quadrature rules.
\begin{corollary}
  Let $y \not\in \varphi_{N-1}^{-1}(0)$. Then the greedy iteration \eqref{eq:Q-greedy-F} with $y_1^{F\ast} = y$ produces the unique positive $L^2_\rho$ quadrature rule abscissae with optimal polynomial accuracy. In particular, if $y \in \varphi_N^{-1}(0)$ then they produce the the abscissae of the $\rho$-weighted Gauss quadrature rule.
\end{corollary}

The theoretical results of this section give strong motivation for using the greedy weighted determinant design $A_N^{F\ast}(Q)$: in one dimension we produce optimally-conditioned point sets for almost any starting location in the greedy design. The one-dimensional greedy designs coincide with the more onerous simultaneous optimization designs, and even coincide with designs based on condition number optimization. For multiple dimensions we retain all the previous properties but must make the assumption that a point set with unit condition number exists, and that our starting location lies in this set. Under this existence assumption, greedy designs again produce optimal sets.

The remainder of this paper investigates the computational performance of the set $A_N^{F\ast}(Q)$.

\section{CFP algorithmic details}
%\subsection{Least-squares with Christoffel-weighted Fekete points}

The CFP strategy given by \eqref{eq:Q-greedy-F} can be used to construct a sample set $A$ having the same size as the dimension of the (weighted) polynomial space $Q$. However, in least-squares problems we wish the sample count $M = |A|$ to dominate the polynomial space dimension $N = \dim Q$. To achieve this, we start with a specified space $Q$ and enrich it with weighted polynomials of a higher degree. This procedure is largely \textit{ad hoc}, so we cannot claim optimality for our specific strategy, but our numerical results indicate that our procedure works very well.

Let $\Lambda$ be a given multi-index set of size $N$. All our examples will use a downward-closed set $\Lambda$, but this is mainly for algorithmic convenience. The weighted polynomial space $Q$ is defined in \eqref{eq:Q-def}. We wish to compute a set of samples $A = \{y_1, \ldots, y_M\}$ that we use to solve the weighted least-squares problem \eqref{eq:weighted-Lsequation}. Our procedure to accomplish this enriches $\Lambda$ to a size of $M > N$, and we subsequently compute $M$ CFP points associated to this enriched $\Lambda$. We then compute the least-squares solution \eqref{eq:weighted-Lsequation} using the original index set $\Lambda$. We describe the details of this below.

\subsection{Choosing $Q$}\label{sec:enrichment}
%Let $\Lambda$ be a specified size-$N$ multi-index set in $\N_0^d$. We assume that $\Lambda$ is downward-closed, mainly for algorithmic convenience. This defines the weighted polynomial space \eqref{eq:Q-def}.
%\begin{align*}
%  Q_\Lambda = \mathrm{span} \left\{ p(x) \sqrt{\lambda(x)} \;\; \big| \;\; p \in P_\Lambda \right\}.
%\end{align*}
Given an enrichment size $\Delta N$, we define $\widetilde{\Lambda}$ of size $N + \Delta N$ and satisfying $\widetilde{\Lambda} \supset \Lambda$ by adding to $\Lambda$ elements based on total-degree graded reverse lexicographic ordering of the set $\Lambda_n \backslash \Lambda$. Precisely:
\begin{enumerate}
  \item Compute $n = \mathrm{max} \left\{ |\alpha| \;\; \big| \;\; \alpha \in \Lambda \right\}$.
  \item Compute $S \coloneqq \Lambda^{\mathrm{TD}}_n \backslash \Lambda$. Set $n \gets n+1$ if $\Lambda = \Lambda^{\mathrm{TD}}_n$.
  \item Impose a total order on $S$: graded (partially-ordered) based on total degree, and ordering within a grade defined by reverse lexicographic ordering.
  \item Extract the first $\Delta N$ elements from $S$ and append those to $\Lambda$, creating the set $\widetilde{\Lambda}$.
\end{enumerate}
The input to the procedure above is simply $\Lambda$ (assumed downward-closed) and an enrichment size $\Delta N$. In all our numerical tests, we choose $\Delta N = \left\lfloor 0.05 N \right\rfloor$, representing a small enrichment of $\Lambda$. We now define the $(N + \Delta N)$-dimensional space
\begin{align*}
  \widetilde{Q} \coloneqq Q_{\widetilde{\Lambda}}.
\end{align*}
We will compute CFP using this enriched weighted space.

\subsection{Choosing candidate sets}
CFP sets are computed via \eqref{eq:Q-greedy-F}. This procedure must be discretized in practice since it is computationally difficult to optimize over the continuum $\Gamma$. We choose a large but finite-size candidate set $\widetilde{A}$, and replace the maximization over $\Gamma$ by maximization over $\widetilde{A}$.

%This section describes considerations in choosing a candidate set $\widetilde{\calA}$ to replace $\Gamma$ in the greedy optimization procedure \eqref{eq:greedy-det-optimization}.
When $d$ is ``moderately small", say $d \leq 5$, modern computational power allows us to choose a candidate mesh that ``fills" this $d$-dimensional space. Thus, we can be reasonably sure that an intelligent choice for $\widetilde{A}$ is an effective surrogate for $\Gamma$.  However when $d \gtrsim 5$ we can no longer be reasonably confident that $\widetilde{A}$ is a fine enough mesh on $\Gamma$. In this paper we do not make any advancements with respect to this deficiency. In our numerical simulations we choose $\widetilde{A}$ as the union of two random ensembles:
\begin{align*}
  \widetilde{A} = \left\{ R_1, \ldots, R_{\widetilde{M}/2}, S_1, \ldots, S_{\widetilde{M}/2} \right\},
\end{align*}
$\widetilde{M}$ is a reasonably large number, usually $\widetilde{M} = 10^4$.  The random variables $R_j$, $j = 1, \ldots, \widetilde{M}/2$ are independent and identically distributed (iid) samples from the probability density $\rho$ on $\Gamma$, and $S_j$, $j=1, \ldots, \widetilde{M}/2$ are iid samples from a degree-asymptotic density inspired by the approach in \cite{Narayan_2014Christoffel}.

When $\rho$ is the uniform measure on $\Gamma = [-1,1]^d$ we choose $S_j$ to be sampled from the tensor-product Chebyshev density, and when $\rho$ is the Gaussian meaure on $\R^d$ we choose the sampling distribution of $S_j$ to have support on the $\R^d$ unit ball with radius $\sqrt{2 n}$, with $n$ as chosen in Section \ref{sec:enrichment}. Its density as a function of $s \in \R^d$ is
\begin{align*}
  C_d \left(1 - \frac{1}{2n} \sum_{k=1}^d s_k^2\right)^{d/2}, \hskip 20pt s = (s_1, \ldots, s_d) \in \R^d,
\end{align*}
where $C_d$ is a normalization constant. We refer to \cite{Narayan_2014Christoffel} for details on this latter sampling density. Our choices above may naturally be replaced with any other candidate mesh, e.g., quasi Monte Carlo sets, sparse grids, or tensor-product points.

We finally note that weakly admissible meshes \cite{Calvi_2008uniformapp} are known to be good candidate sets for \textit{non}-weighted polynomial spaces. These meshes are good candiates for our weighted formulation as well, but we forgo their use since known constructions of these meshes exhibit very large growth for even moderate dimensions \cite{Bos_2011Geometric}.

\subsection{Greedy optimization}\label{sec:greedy}
We have chosen a space $\widetilde{Q}$ of dimension $M = N + \Delta N$ and a candidate mesh $\widetilde{A}$ of size $\widetilde{M} \gg N$. Our goal is now to compute a CFP set:
\begin{align*}
  \textrm{Solve \eqref{eq:Q-greedy-F}, setting $Q \gets \widetilde{Q}$ and $\Gamma \gets \widetilde{A}$}
\end{align*}
This optimization procedure at each iteration is equivalent to forming the matrix $\bs{V}\left(\widetilde{\calA}, \widetilde{Q}\right)$ and then to choose rows that greedily maximize the spanned volume of the chosen rows. This, in turn, is easily performed by a column-pivoted QR decomposition of $\bs{V}^T$, see, e.g., \cite{sommariva_computing_2009,bos_computing_2010}. The ordered pivots define the choice of $M$ points $A_M = A_M^{F\ast}$. %Therefore, the optimization above can be solved via relatively simple linear algebraic operations. The produced point set $\calA^\dagger$ has $M \coloneqq N + \delta N$ points.

\subsection{Least-squares solve}
Having selected a size-$M$ point set $A_M$ in Section \ref{sec:greedy}, and with a size-$N$ polynoimal space $P(\Lambda)$ already defined (along with its weighted version $Q$), we now compute $\bs{\widehat{f}}$ as the solution to \eqref{eq:weighted-Lsequation}.

The coefficients in the solution vector $\bs{\widehat{f}}$ define our desired expansion shown in \eqref{eq:f-exp}.

\section{Numerical tests} \label{sec:tests}
In this section we investigate the stability and convergence properties of the CFP sampling strategy. We will generate sampling sets $A_M$ using CFP and compare it against two popular alternatives for generating $A_M$: randomized Monte Carlo using independent and identically-distributed samples from $\rho$, and deterministic Approximate Fekete points. We are interested primarily in investigating how sampling rates of $M$ versus the approximation space dimension $N$ affects stable and accurate approximate. In our figures and results, we will use ``MC" to denote  Monte Carlo procedure, ``Fekete" to denote approximate Fekete points, and ``C-Fekete" to denote the CFP algorithm of this paper.

In order to implement our proposed method, we first picked $\widetilde{M}$ random points as a candidate set, then we select $M$ optimal points from Section \ref{sec:methodology} as CFP points. In all the following examples, we choose $\widetilde{M}=10^4$.

\subsection{Matrix stability}

In this section we investigate the condition number of the matrix $\bs{V}(A_M,P)$ (MC and Fekete), and the condition number for the matrix $\bs{V}(A_M,Q)$ (C-Fekete).  In all examples that follow we perform 50 trials of each procedure and report the mean condition number along with $20\%$ and $80\%$ quantiles.

\subsubsection{Bounded domains}
We first consider the Legendre polynomials for which the domain is $\Gamma=[-1,1]^d$ and $\rho$ is the uniform probability density.  In Fig. \ref{fig:d2_Leg_cond}, the $d=2$ condition numbers obtained by the procedures are shown for a linear over-sampling of $M=1.05N$. The numerical results for $d=6$ and $d=10$ are shown in Fig. \ref{fig:d6_Leg_cond} and \ref{fig:d10_Leg_cond}, respectively. We notice that the CFP procedure produces point configurations that have notably more stable linear systems when compared against AFP or MC designs.

\begin{figure}[htbp]
\begin{center}
    \includegraphics[width=6cm]{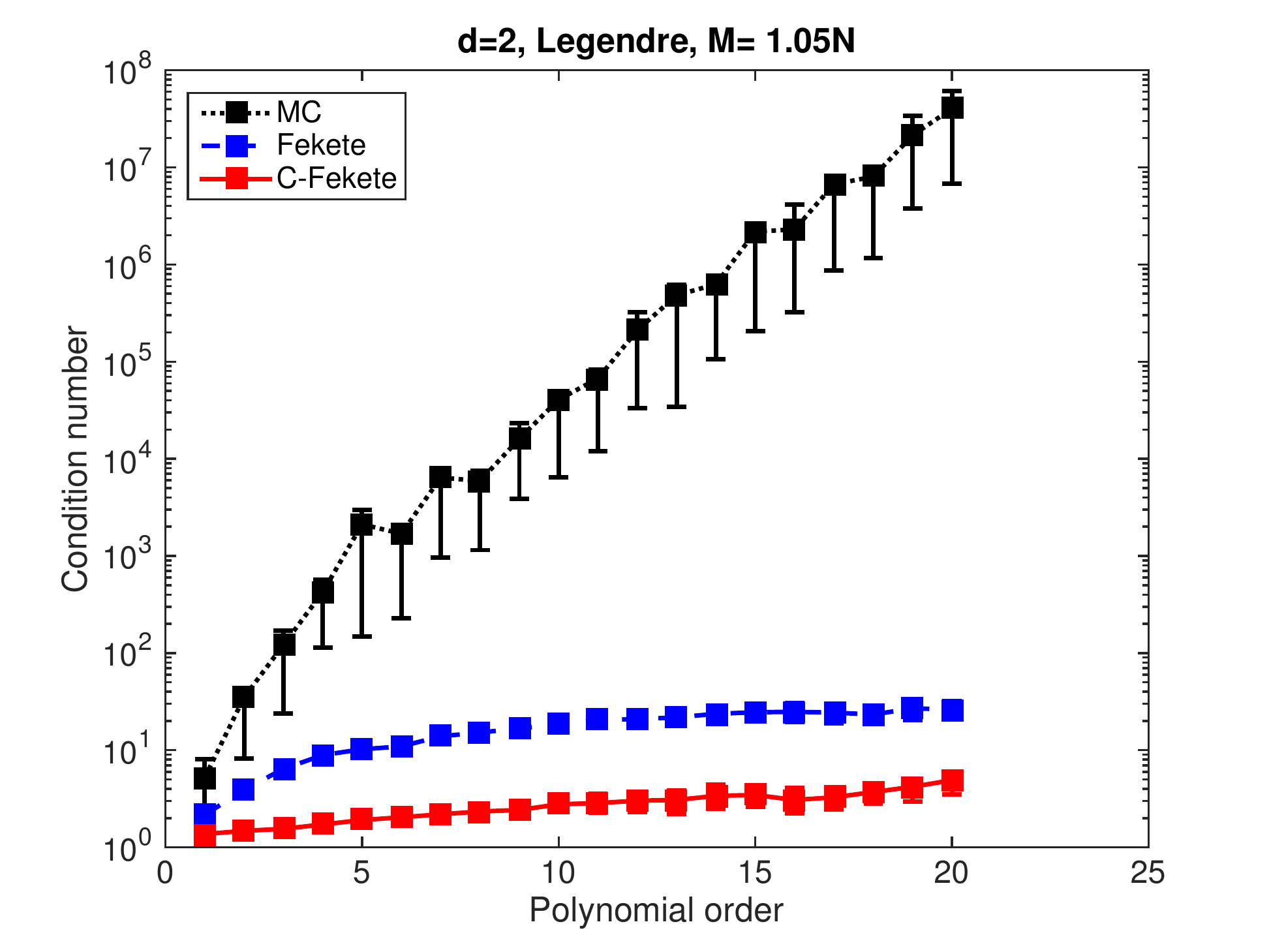}
    \includegraphics[width=6cm]{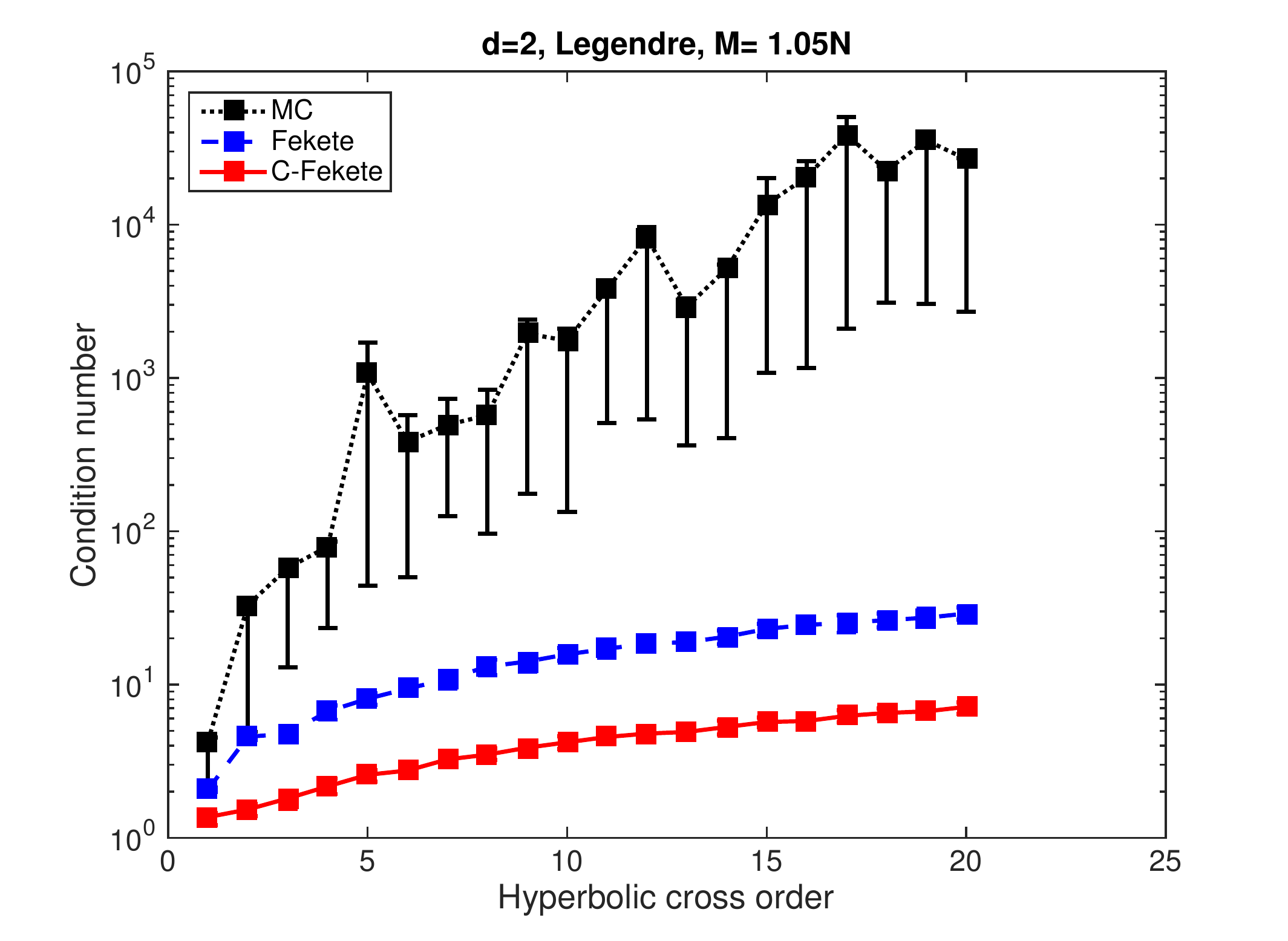}
\end{center}
  \caption{Condition number with respect to the polynomial degree in the 2-dimensional polynomial spaces. Left: Total degree (TD); Right: Hyperbolic cross (HC).
  \label{fig:d2_Leg_cond}
    }
\end{figure}

\begin{figure}[htbp]
\begin{center}
    \includegraphics[width=6cm]{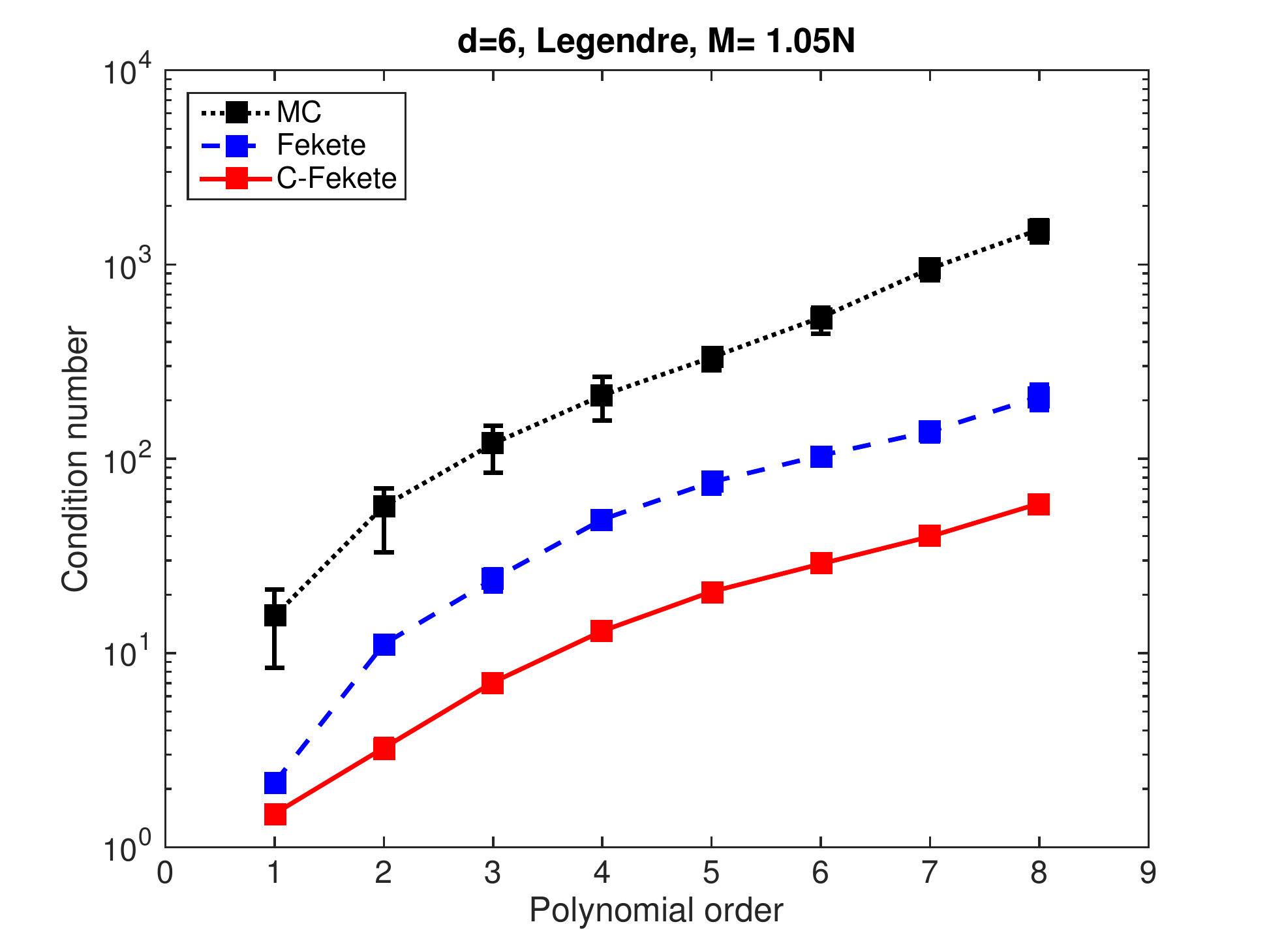}
    \includegraphics[width=6cm]{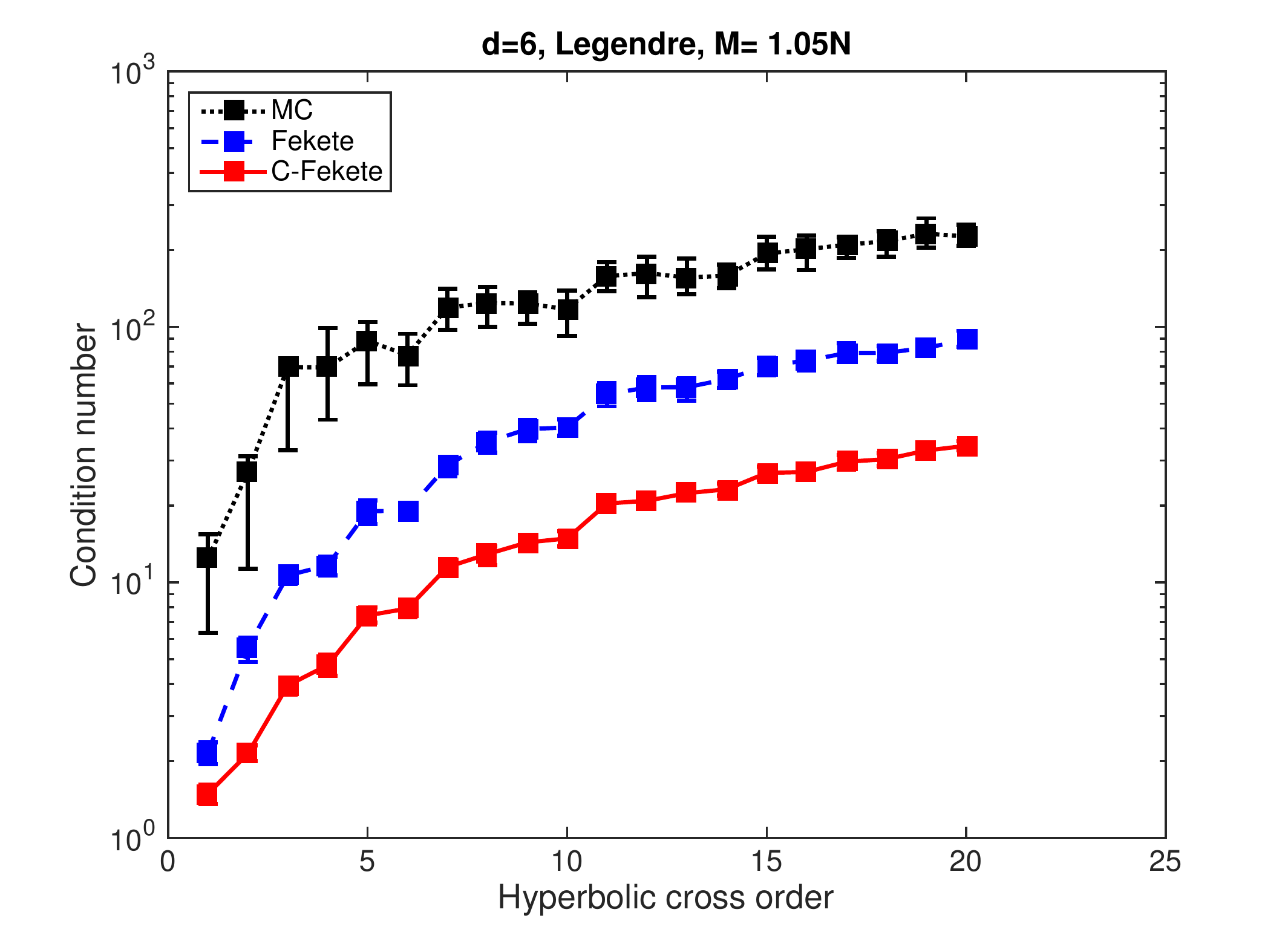}
\end{center}
  \caption{Condition number with respect to the polynomial degree in the 6-dimensional  polynomial spaces. Left: Total degree (TD); Right: Hyperbolic cross (HC).
  \label{fig:d6_Leg_cond}
    }
\end{figure}

\begin{figure}[htbp]
\begin{center}
    \includegraphics[width=6cm]{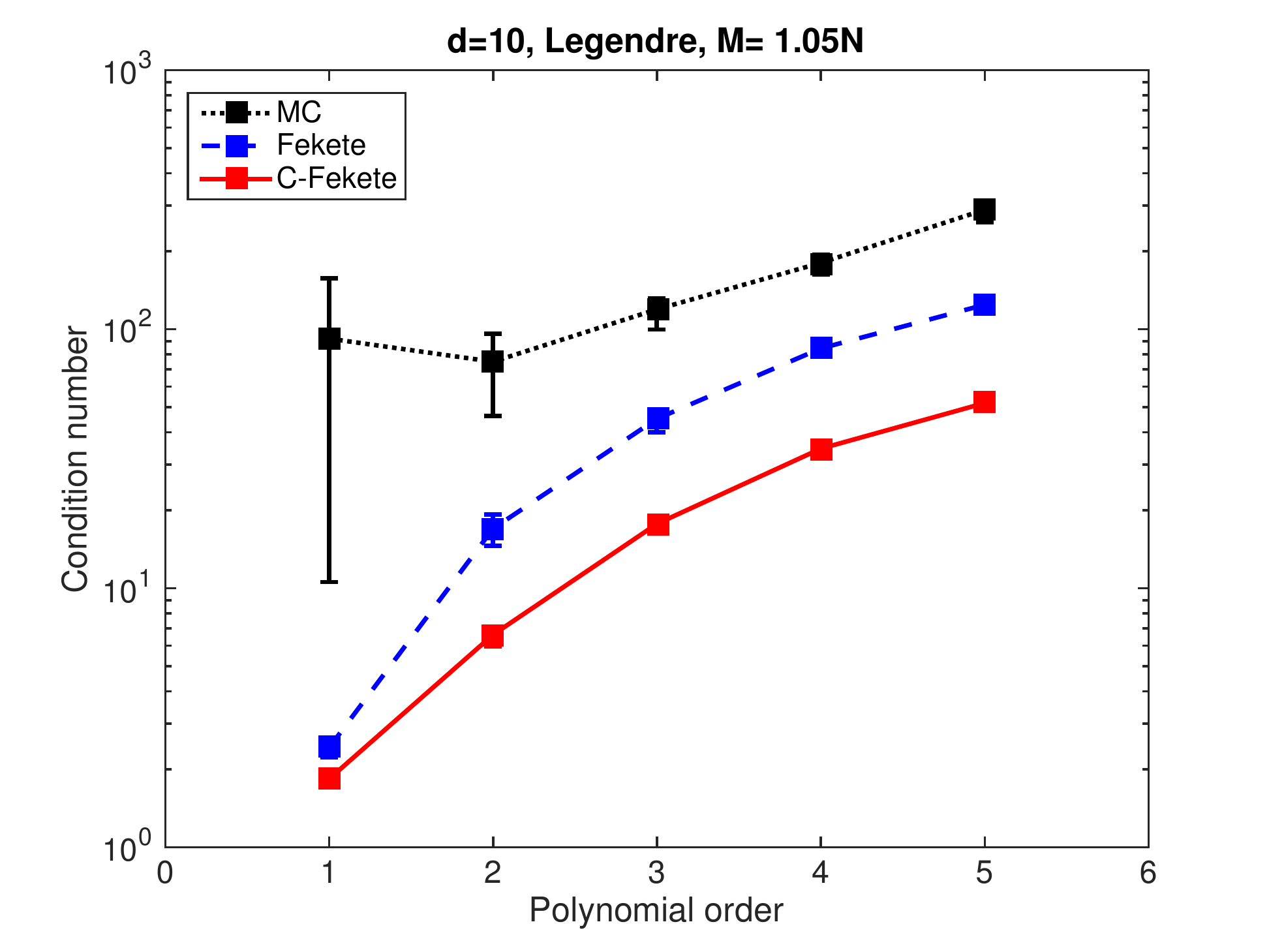}
    \includegraphics[width=6cm]{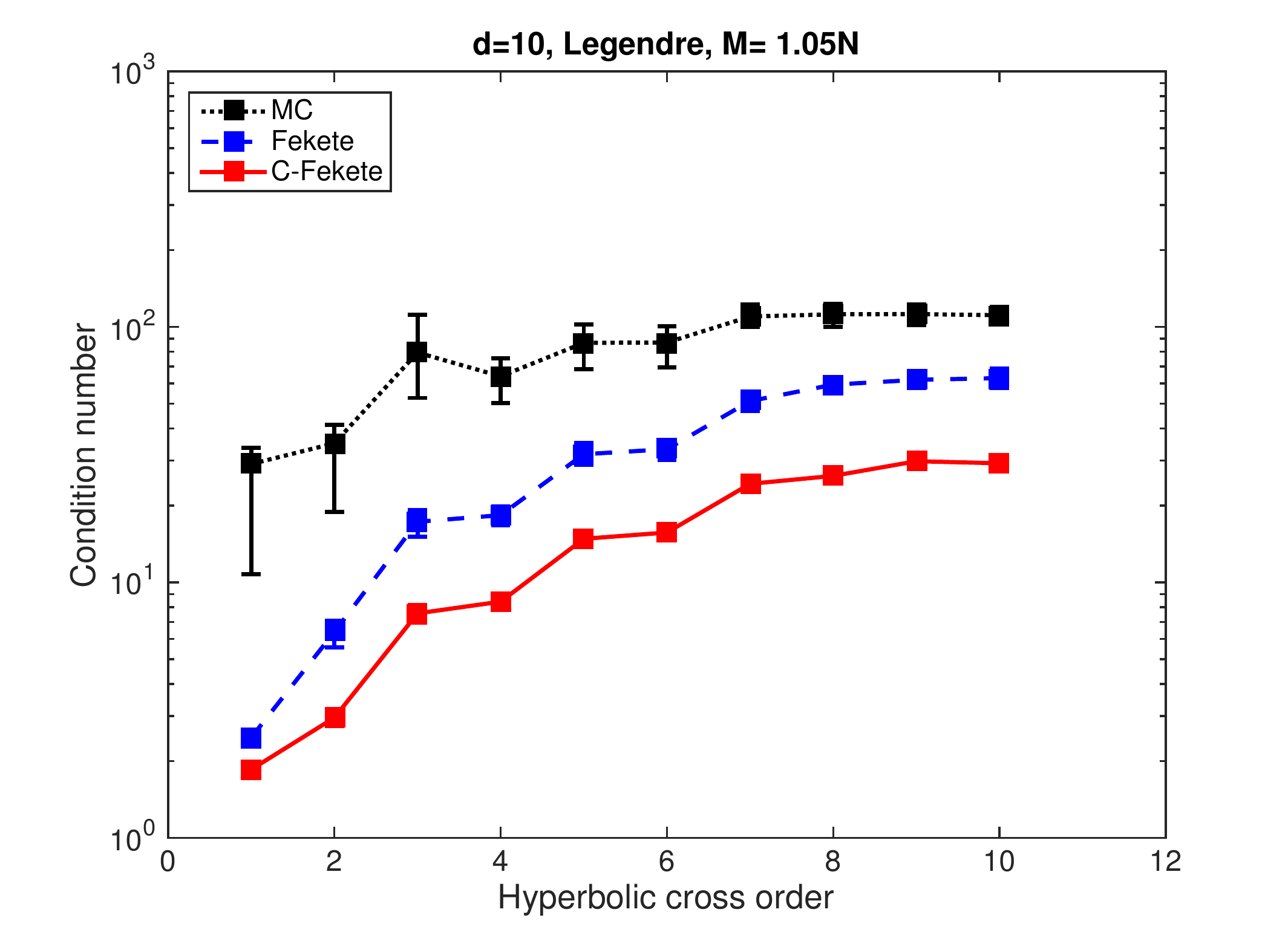}
\end{center}
  \caption{Condition number with respect to the polynomial degree in the 10-dimensional  polynomial spaces. Left: Total degree (TD); Right: Hyperbolic cross (HC).
  \label{fig:d10_Leg_cond}
    }
\end{figure}

\subsubsection{Unbounded domains}
We now let $\Gamma = \R^d$ with $\rho$ a Gaussian density, $\rho(y) \propto \exp(-\|y\|_2^2)$. The associated orthonormal polynomial family is formed from tensor-product Hermite polynomials. Our tabulation of the condition numbers are shown in Figs. \ref{fig:d2_Her_cond} - \ref{fig:d25_Her_cond}. We note that for each case, dimension $d=2, 6, 10$, and $25$, the CFP procedure produces more stable point sets than either of the alternatives, but the improvement is modest in high dimensions.
%%%%%%%%%%%%%%%%%%%%%%%%%%%%%%%%%%%%%%%%%%%%%%%%%%%%%%%%%%%%%%
\begin{figure}[htbp]
\begin{center}
   \includegraphics[width=6cm]{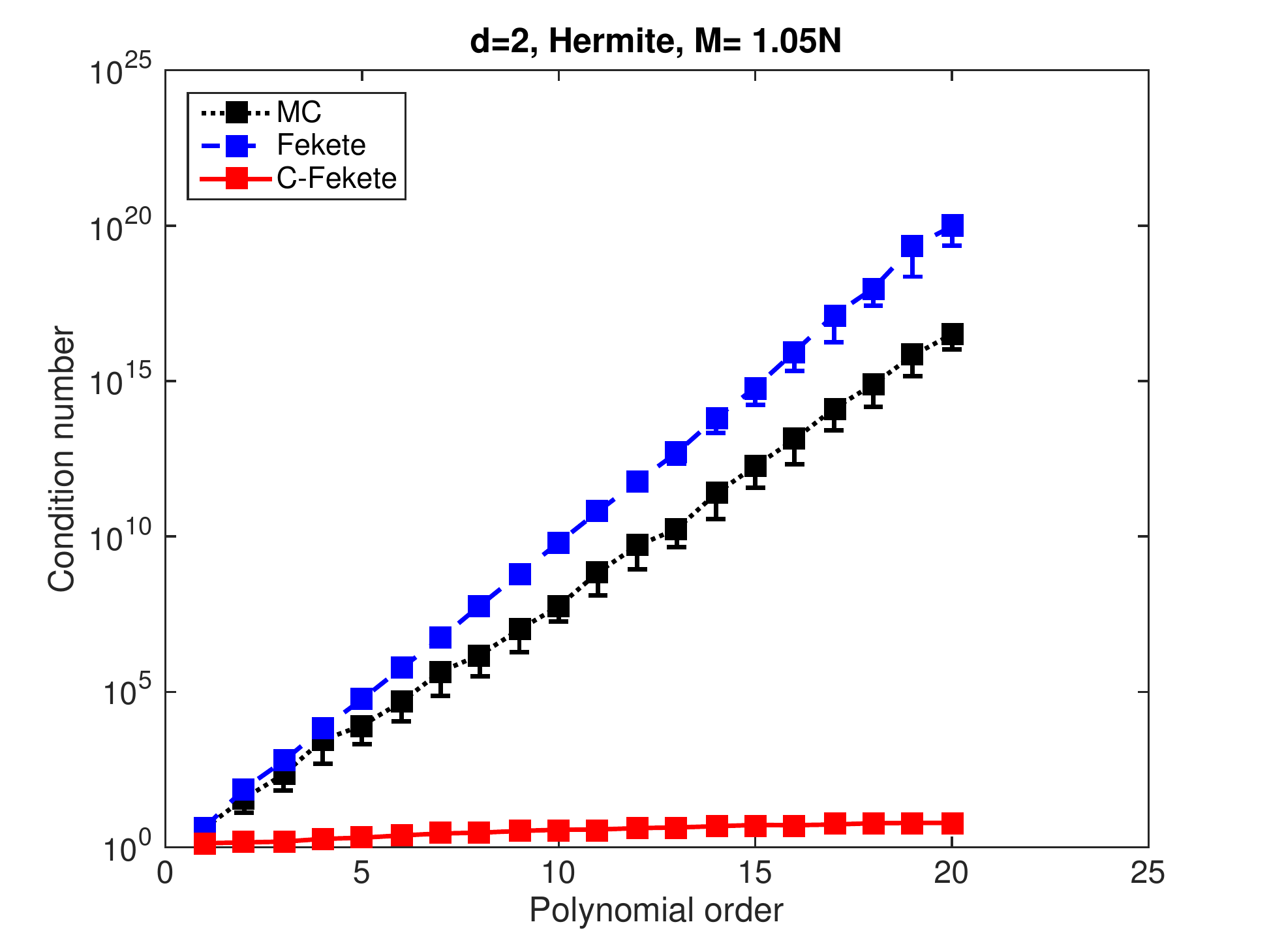}
   \includegraphics[width=6cm]{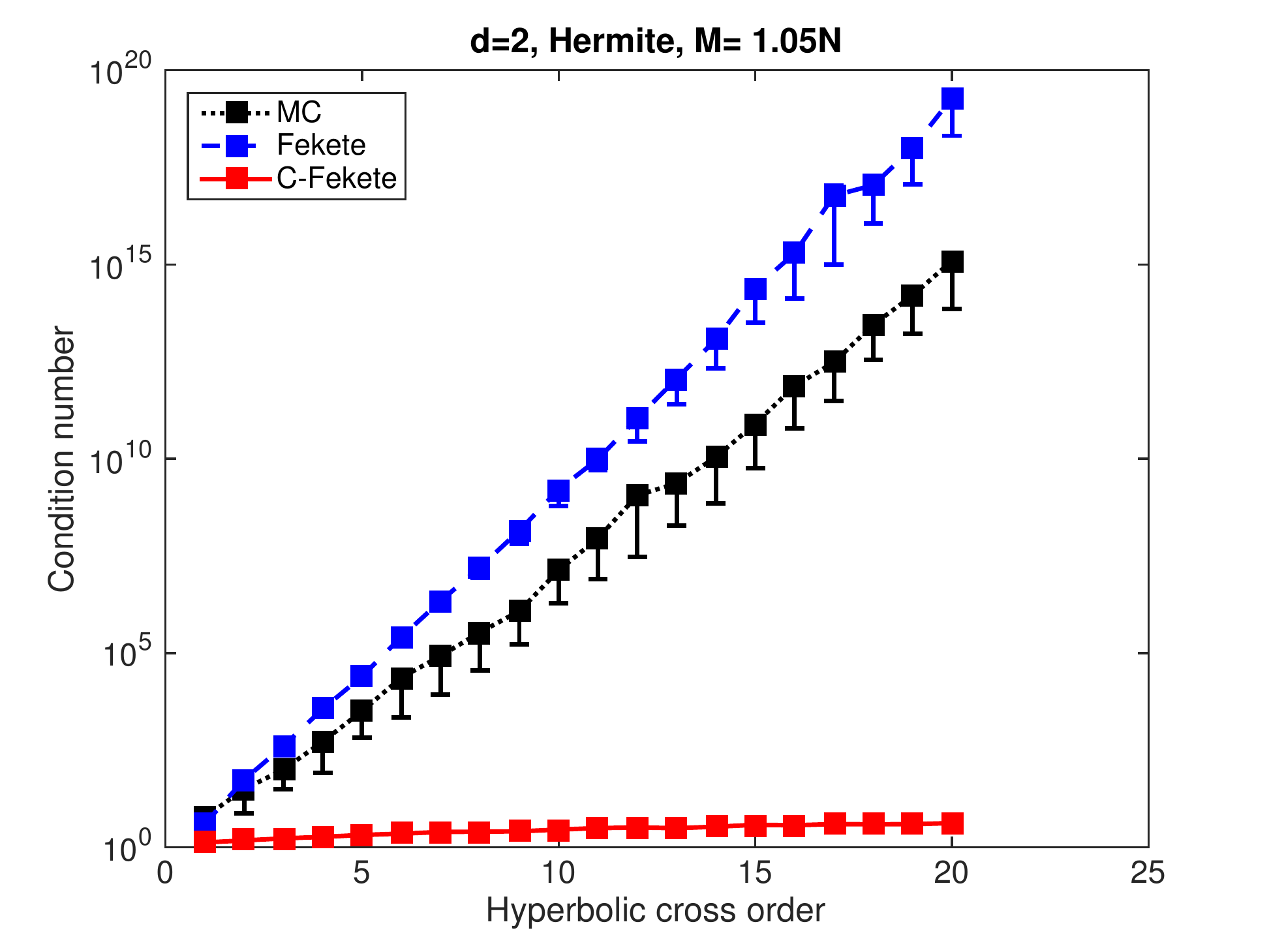}
\end{center}
  \caption{Condition number with respect to the polynomial degree in the 2-dimensional  polynomial spaces. Left: Total degree (TD); Right: Hyperbolic cross (HC).
  \label{fig:d2_Her_cond}
    }
\end{figure}

\begin{figure}[htbp]
\begin{center}
   \includegraphics[width=6cm]{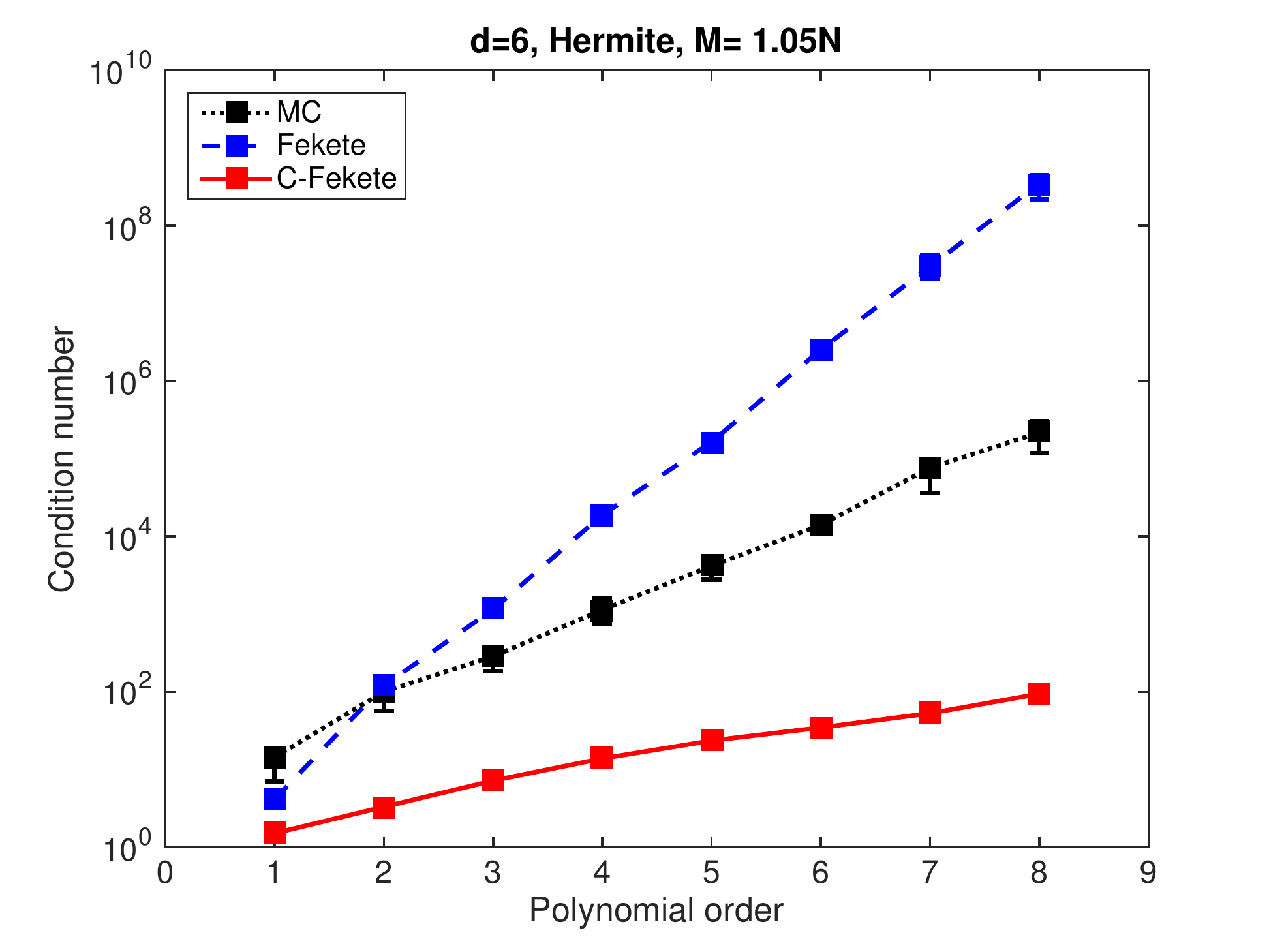}
   \includegraphics[width=6cm]{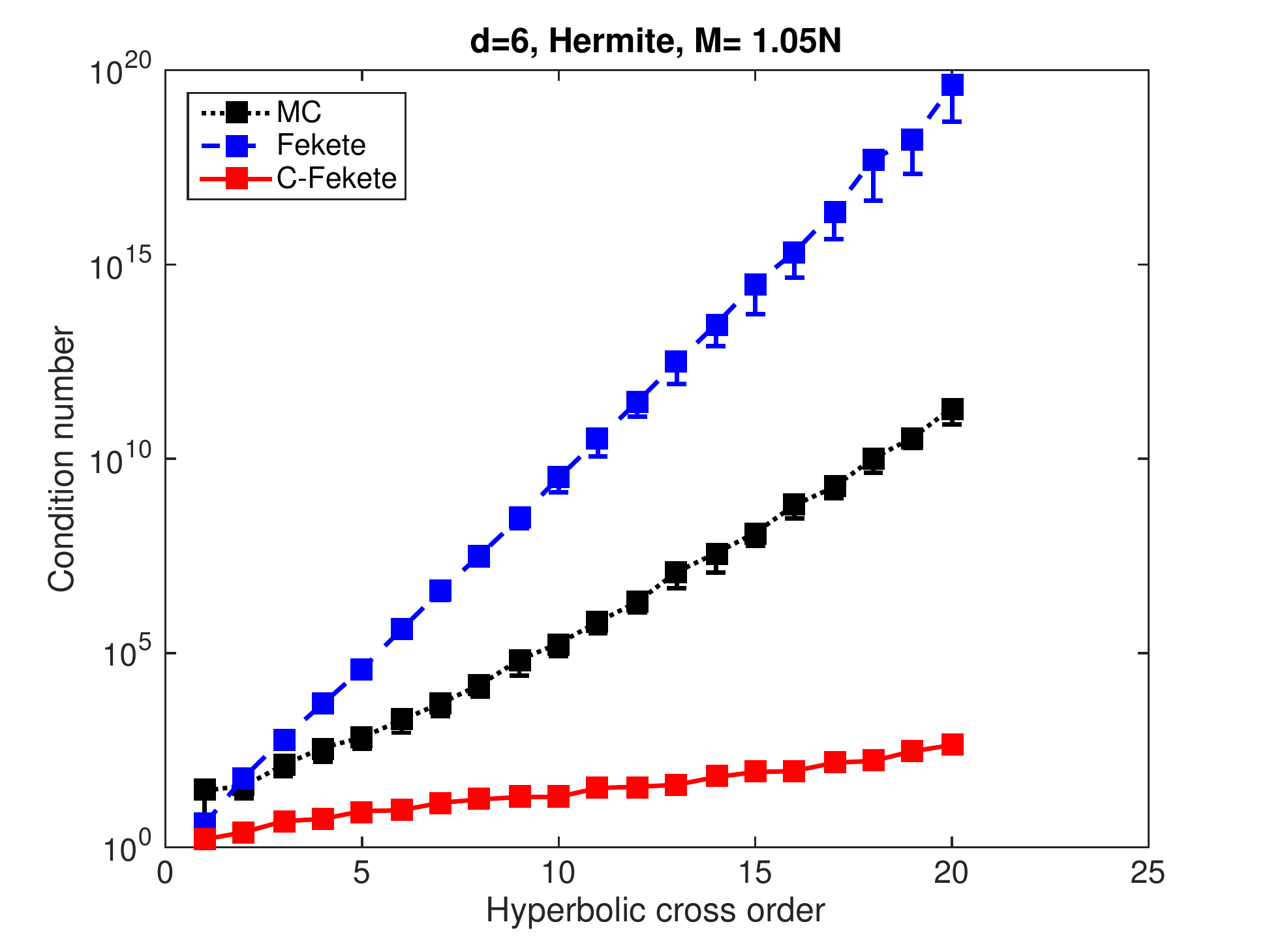}
\end{center}
  \caption{Condition number with respect to the polynomial degree in the 6-dimensional polynomial spaces. Left: Total degree (TD); Right: Hyperbolic cross (HC).
  \label{fig:d6_Her_cond}
    }
\end{figure}

\begin{figure}[htbp]
\begin{center}
   \includegraphics[width=6cm]{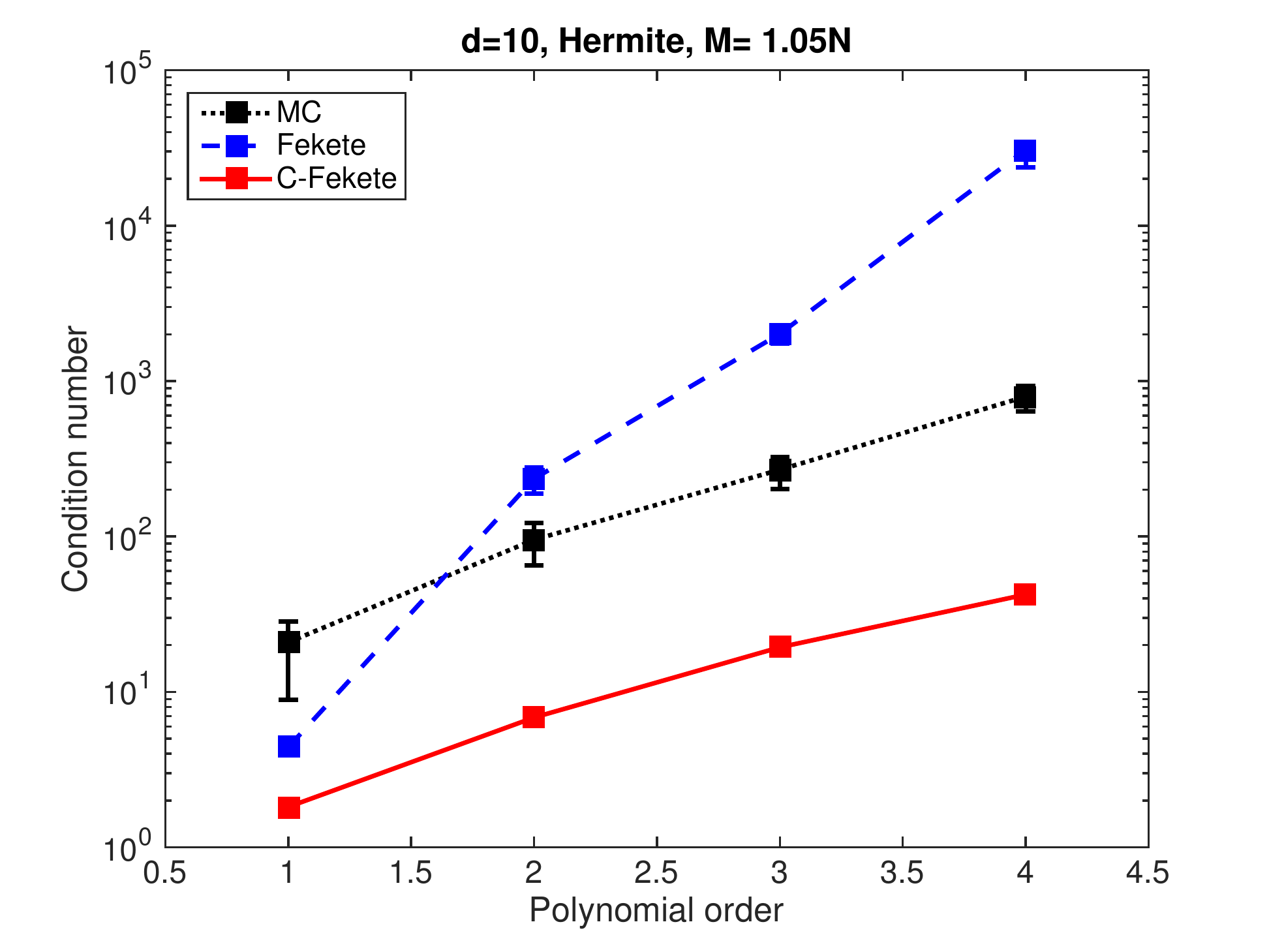}
   \includegraphics[width=6cm]{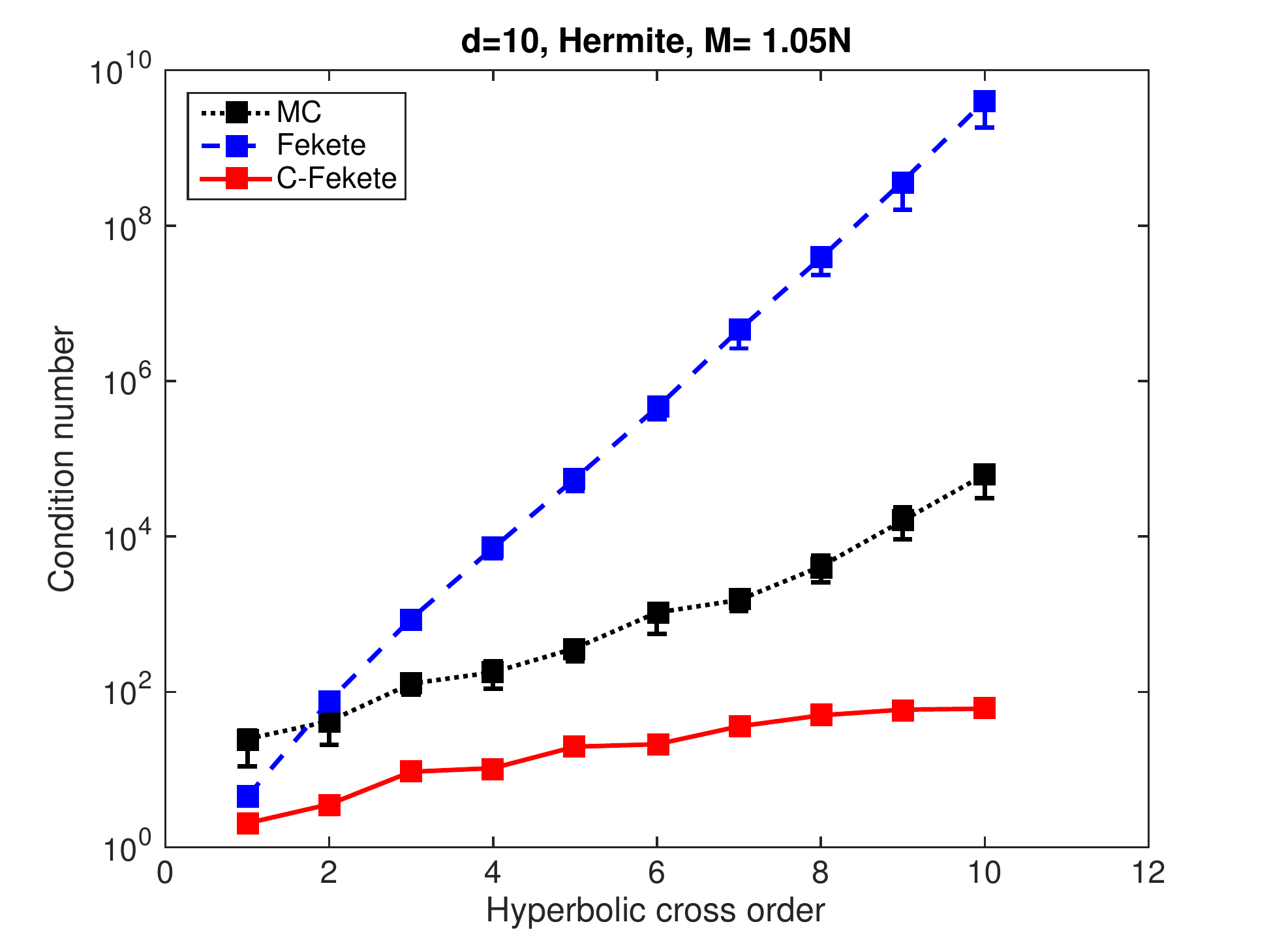}
\end{center}
  \caption{Condition number with respect to the polynomial degree in the 6-dimensional polynomial spaces. Left: Total degree (TD); Right: Hyperbolic cross (HC).
  \label{fig:d10_Her_cond}
    }
\end{figure}

\begin{figure}[htbp]
\begin{center}
   \includegraphics[width=6cm]{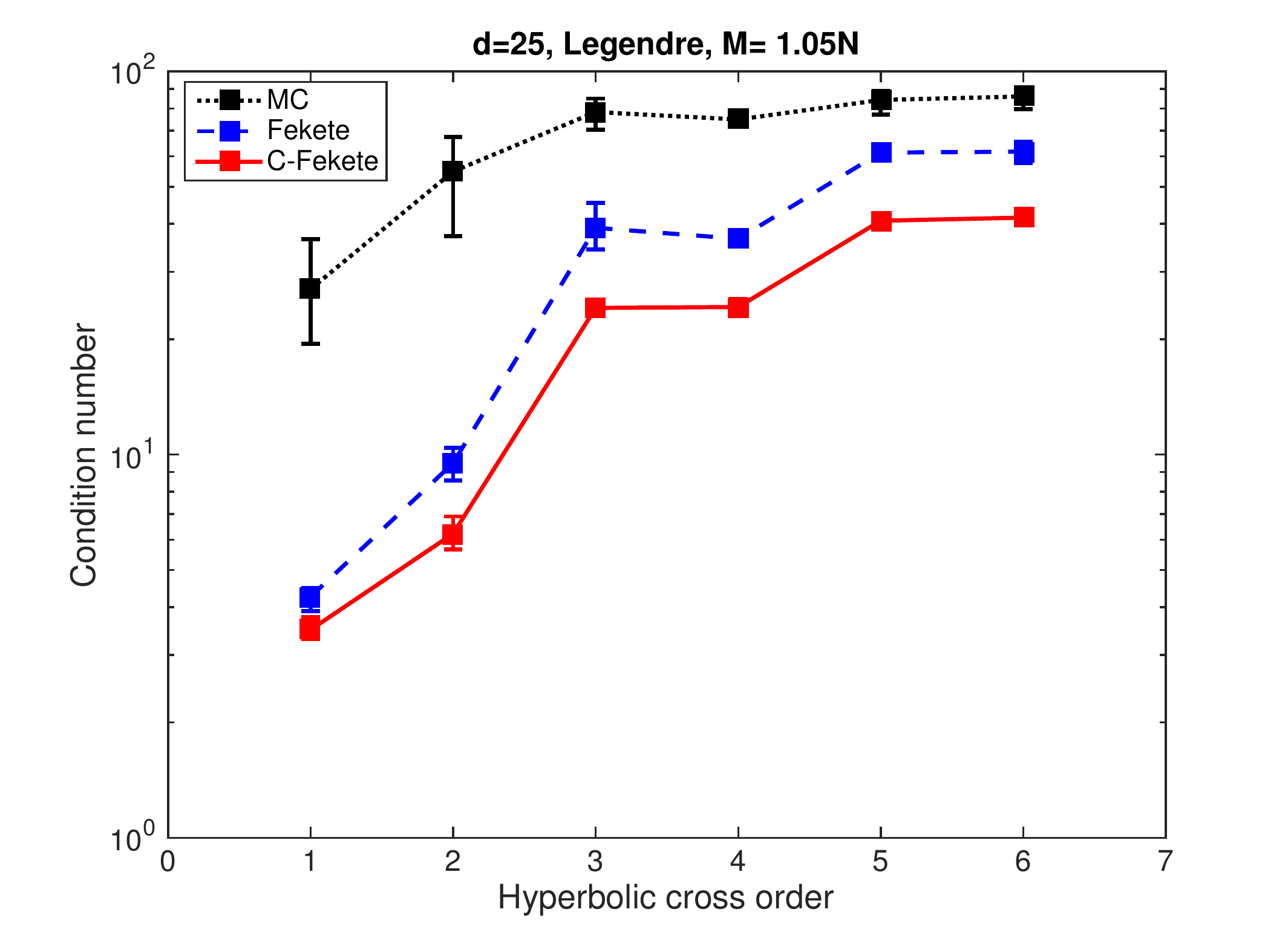}
   \includegraphics[width=6cm]{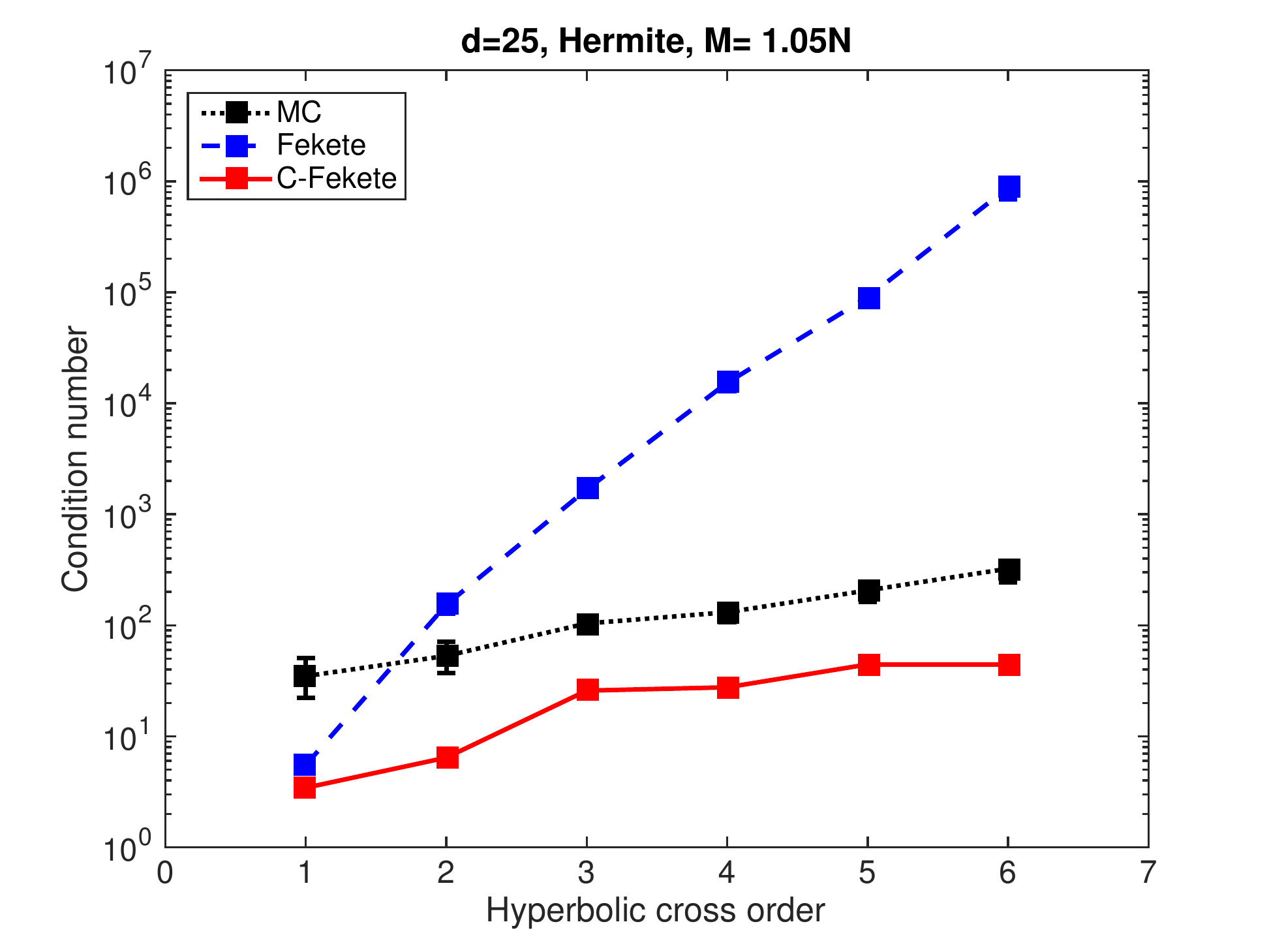}
\end{center}
  \caption{Condition number with respect to the polynomial degree in the 25-dimensional hyperbolic cross polynomial spaces. Left: Legendre; Right: Hermite.
  \label{fig:d25_Her_cond}
    }
\end{figure}

\subsection{Least-squares accuracy}
In this section, we will compare the CFP, APF, and MC algorithms in terms of their ability to approximate test functions. In all examples that follow we report the numerical results over an ensemble of 50 tests.

\subsubsection{Algebraic function} In Fig. \ref{fig:d2_fun_le}, we show the convergence rate of the least-squares approximation for Legendre approximation ($\rho$ uniform over $[-1,1]^d$) in the 2-dimensional polynomial space, for the test function $f(y) = \exp(-\sum^d_{j=1}y_j^2)$. We measure accuracy using the discrete $\ell_2$ norm which is computed using $1,000$ random samples drawn from the probability measure of orthogonality. In this case the CFP and AFP procedures work comparably.
%
%In Fig.\ref{fig:d2_frankefun}, we examine the proposed approach by using 2d Franke function
% \begin{eqnarray} \label{frankefun}
% \begin{split}
%f(\mathbf{Y})=& \frac{3}{4}e^{-((9Y_1-2)^2+(9Y_2-2)^2)/4}+\frac{3}{4}e^{-(9Y_1+1)^2/49-(9Y_2+1)^2)/10}\\
%&+\frac{1}{2}e^{-((9Y_1-7)^2+(9Y_2-3)^2)/4}-\frac{1}{5}e^{-(9Y_1-4)^2-(9Y_2-7)^2}.
%\end{split}
% \end{eqnarray}

In Fig. \ref{fig:d2_fun_he}, we consider the Hermite approximation for $f(y)$ in the 2-dimensional  polynomial  space. We observe here that CFP produces considerably better results compared with MC or AFP, especially for high-degree approximations.

\begin{figure}[htbp]
\begin{center}
  \includegraphics[width=6cm]{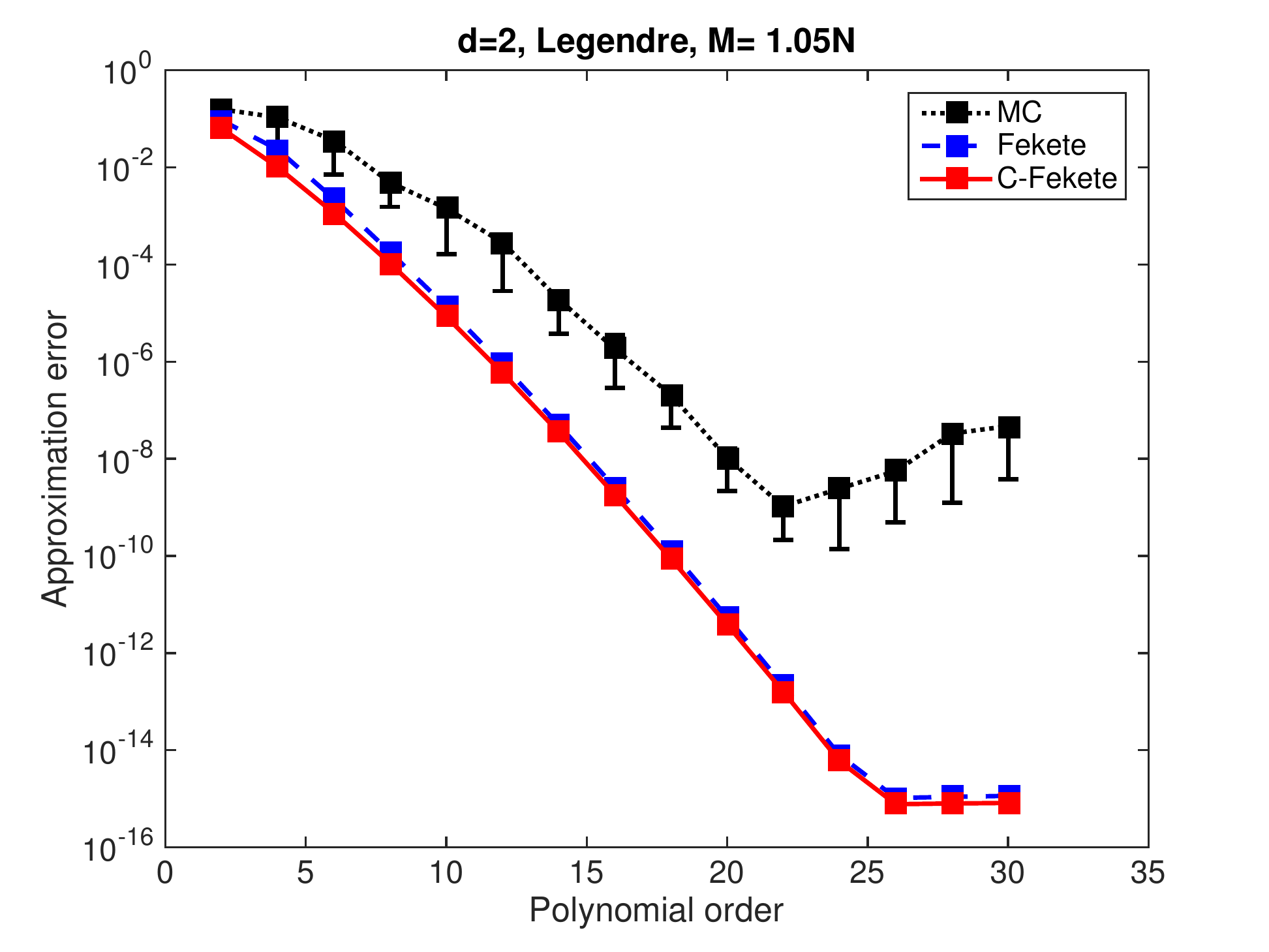}
  \includegraphics[width=6cm]{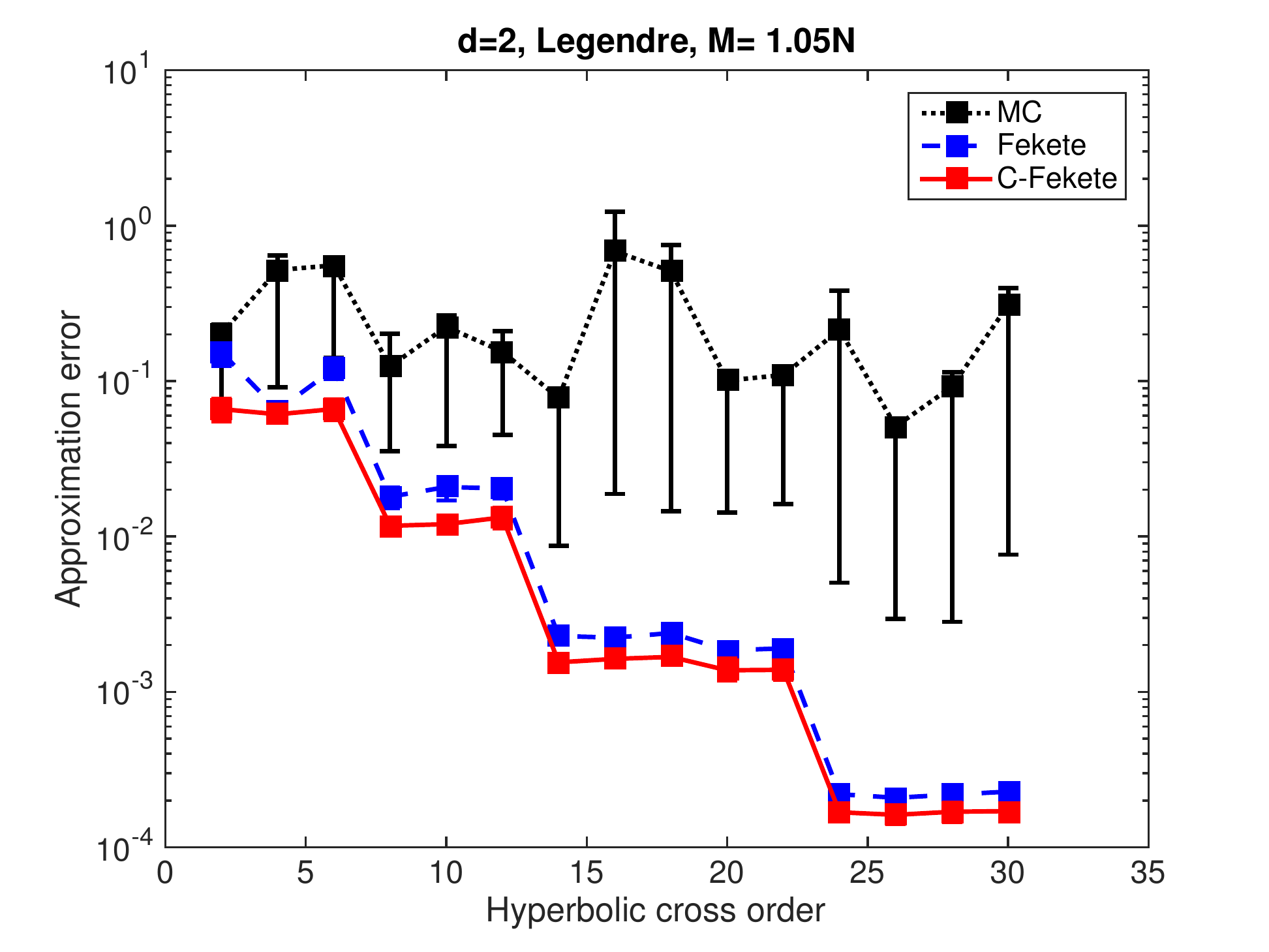}
\end{center}
  \caption{Approximation error against polynomial degree. Legendre approximation of $f(\mathbf{Y})= \exp(-\sum^d_{i=1} Y_i^2)$ Left: TD. Right: HC .
   \label{fig:d2_fun_le}
    }
\end{figure}

\begin{figure}[htbp]
\begin{center}
  \includegraphics[width=6cm]{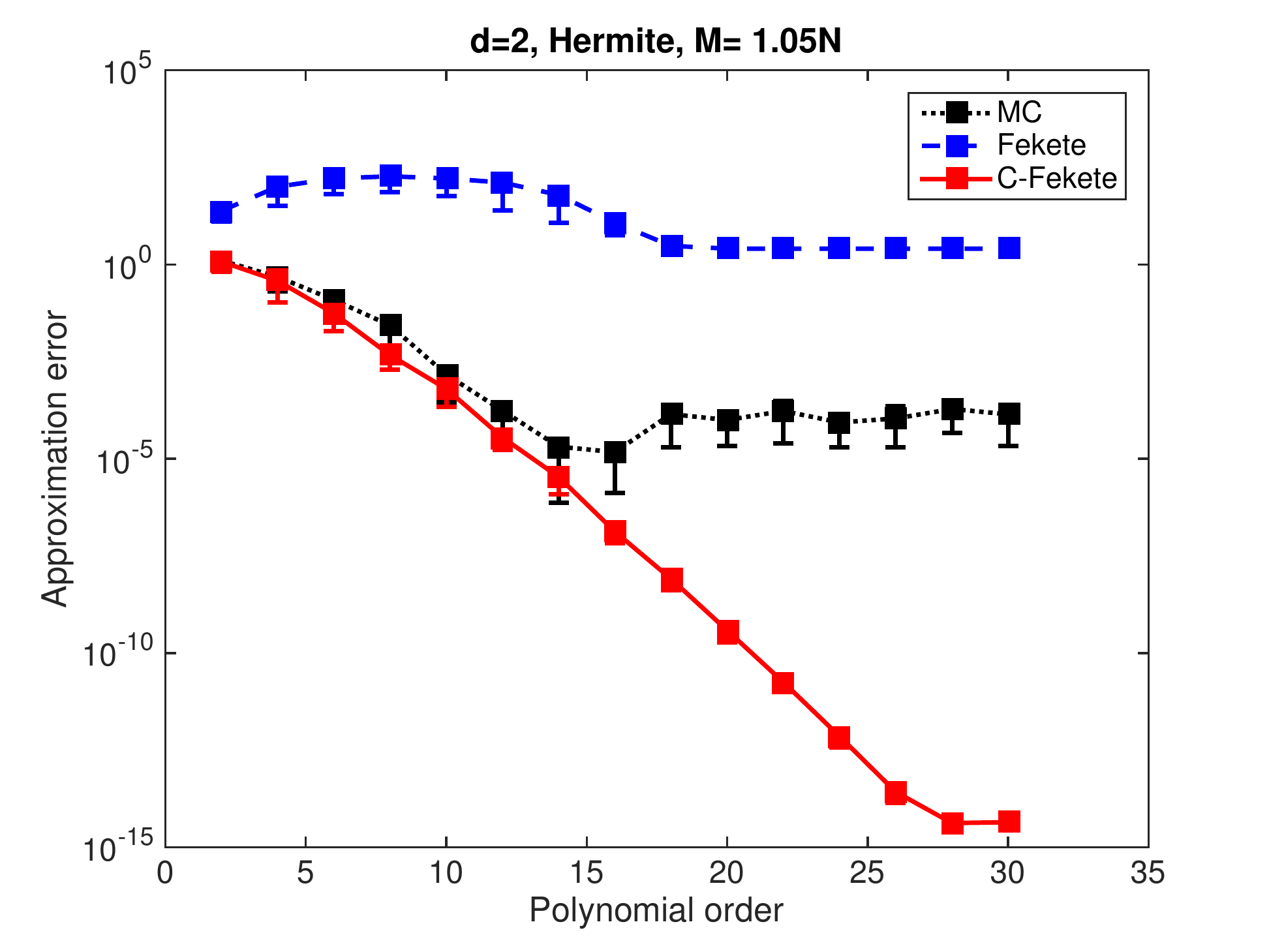}
  \includegraphics[width=6cm]{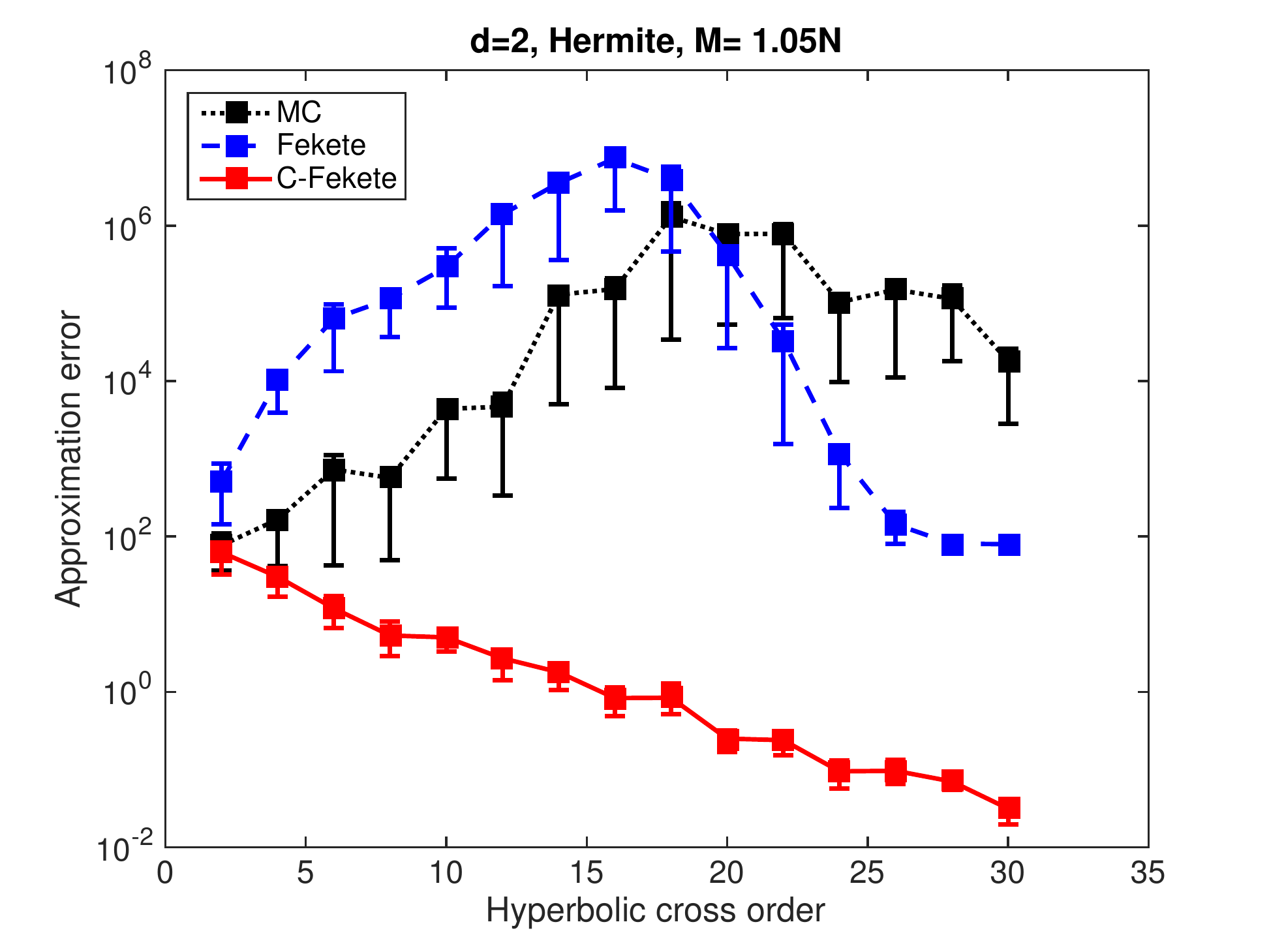}
\end{center}
  \caption{Approximation error against polynomial degree. Hermite approximation of $f(\mathbf{Y}) = \exp(-\sum^d_{i=1} Y_i)$ Left: TD. Right:  HC.
   \label{fig:d2_fun_he}
    }
\end{figure}

%%%%%%%%%%%%%%%%%%%%%%%%%%%%%%%%%%%%%%%%%%%%%%%%%%%%%%%%%%%%%%

  %%%%%%%%%%%%%%%%%%%%%%%%
\subsubsection{Parameterized elliptic differential equation}

We now consider the stochastic elliptic equation, one of the most used
benchmark problems in UQ, in one spatial dimension,
 \begin{equation} \label{spde}
 -\frac{d}{dx} [\kappa(x,y)\frac{du}{dx}(x,y)]=f, \quad  (x,y)\in (0,1) \times \mathbb{R}^d,
 \end{equation}
 with boundary conditions
 \begin{eqnarray*}
 u(0,y)=0,  \quad u(1,y)=0,
 \end{eqnarray*}
and $f=2$. The random diffusivity takes the following form
 \begin{equation}\label{randdiff}
 \kappa(x,y)=1+\sigma \sum^d_{k=1} \frac{1}{k^2 \pi^2} \cos(2\pi k x) y_{k}.
 \end{equation}
 %where $\mathbf{Y}$ is a random vector with independent
 %and identically distributed components.
 This form resembles that of
 the well known Karhunen-Loeve expansion.

 We approximate the solution $u(y)=u(0.5,y)$, and let the density $\rho(y)$ be uniform over the hypercube $[-1,1]^d$, and thus use Legendre polynomials in $y$ to approximate $u(y)$. Fig.\ref{fig:d2_Le_ell} and Fig.\ref{fig:d25_Le_ell} compare the convergence accuracy of the least squares approximations of the quantity of interest $u(0.5,y)$ using CFP, AFP, and MC algorithms. CFP performs comparably, but no worse, than AFP, and much better than MC.

%%%%%%%%%%%%%%%%%%%%%%%%%%%%%%%%%%%%%%%%%%%%%%%%%%%%%%%%%%%%%%
\begin{figure}[htbp]
\begin{center}
  \includegraphics[width=6cm]{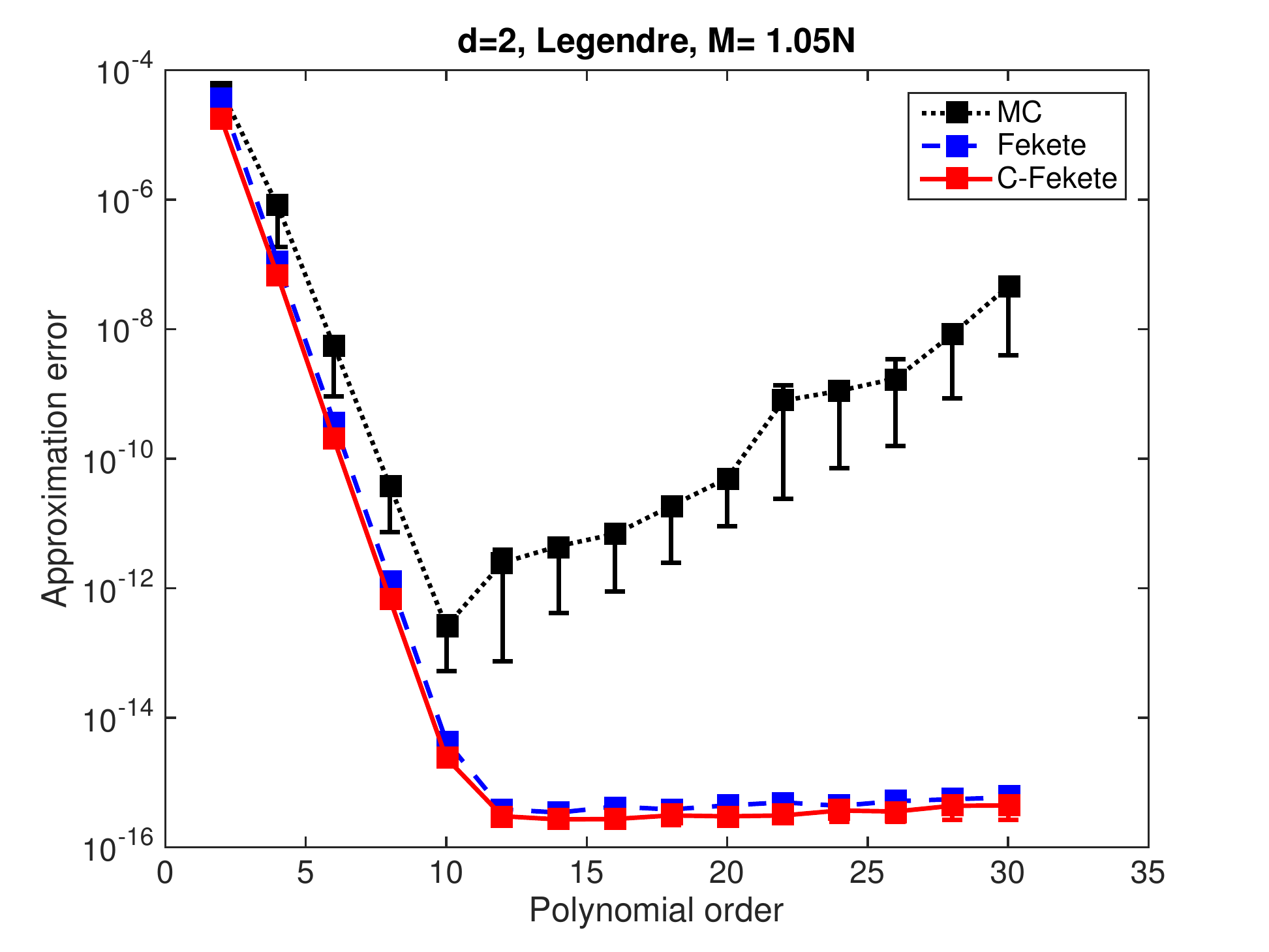}
  \includegraphics[width=6cm]{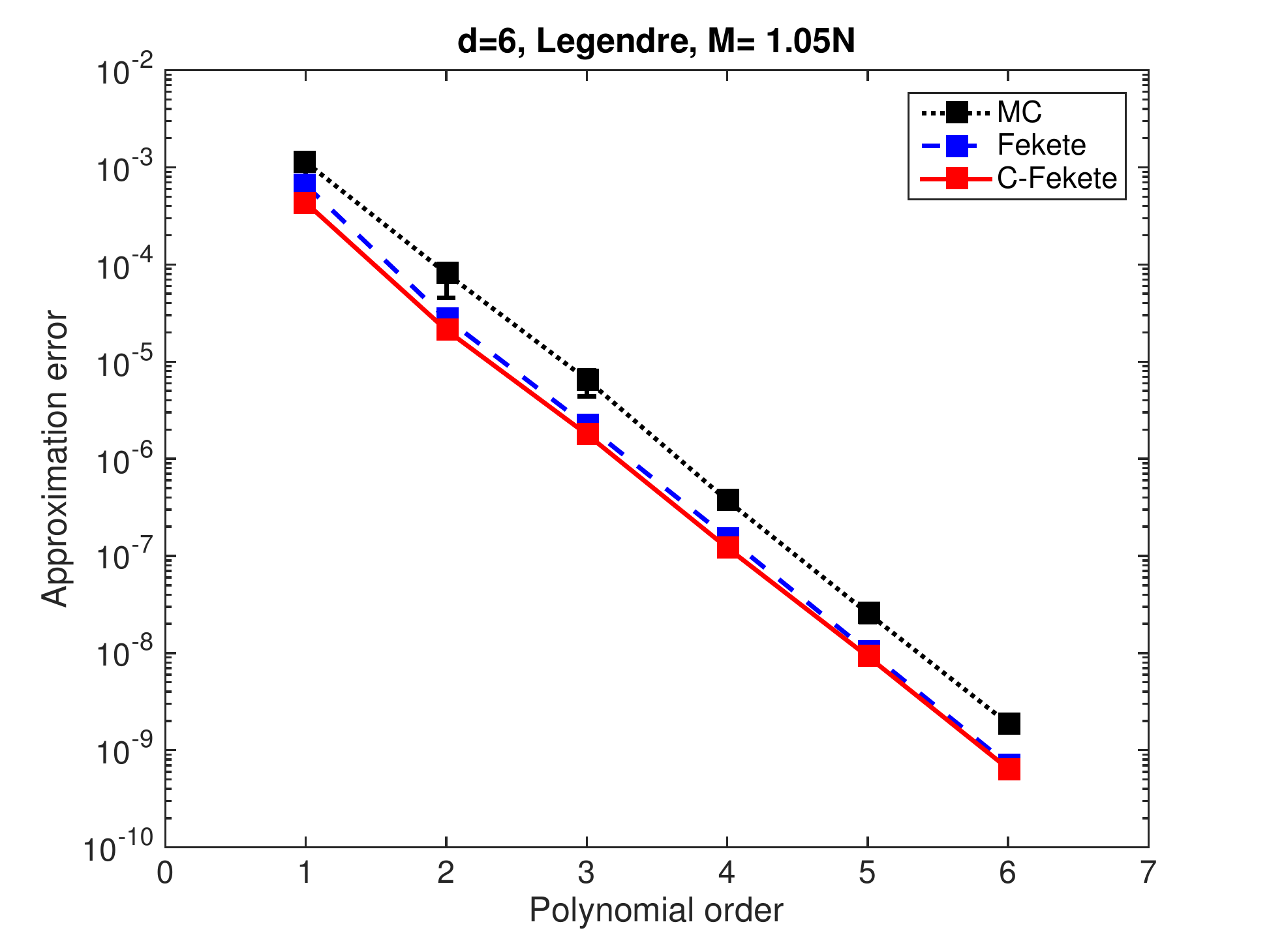}
\end{center}
  \caption{Approximation error against polynomial degree. Legendre approximation of the  diffusion equation. Left: $d=2$; Right: $d=6$.
    \label{fig:d2_Le_ell}
    }
\end{figure}

\begin{figure}[htbp]
\begin{center}
  \includegraphics[width=6cm]{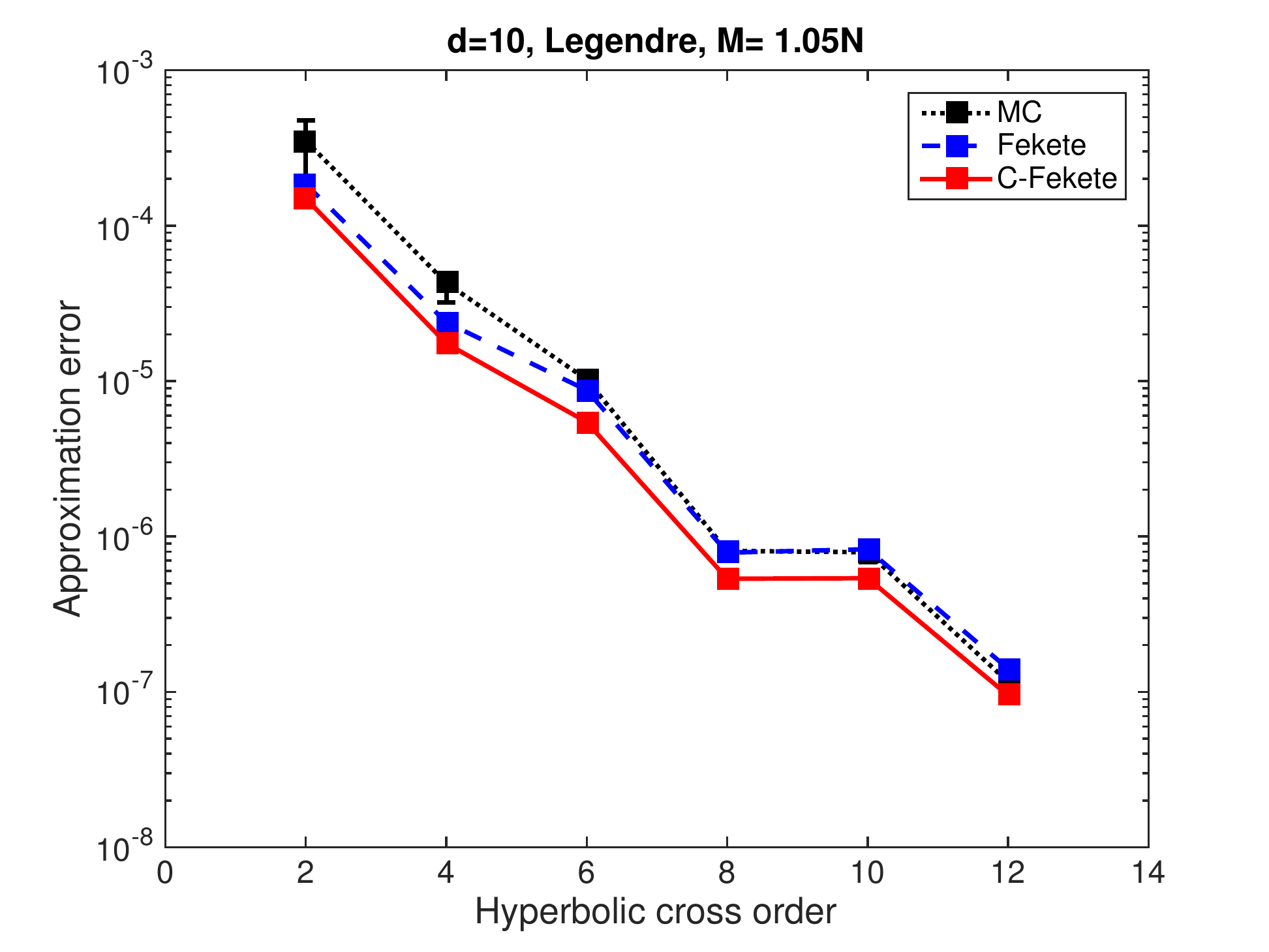}
  \includegraphics[width=6cm]{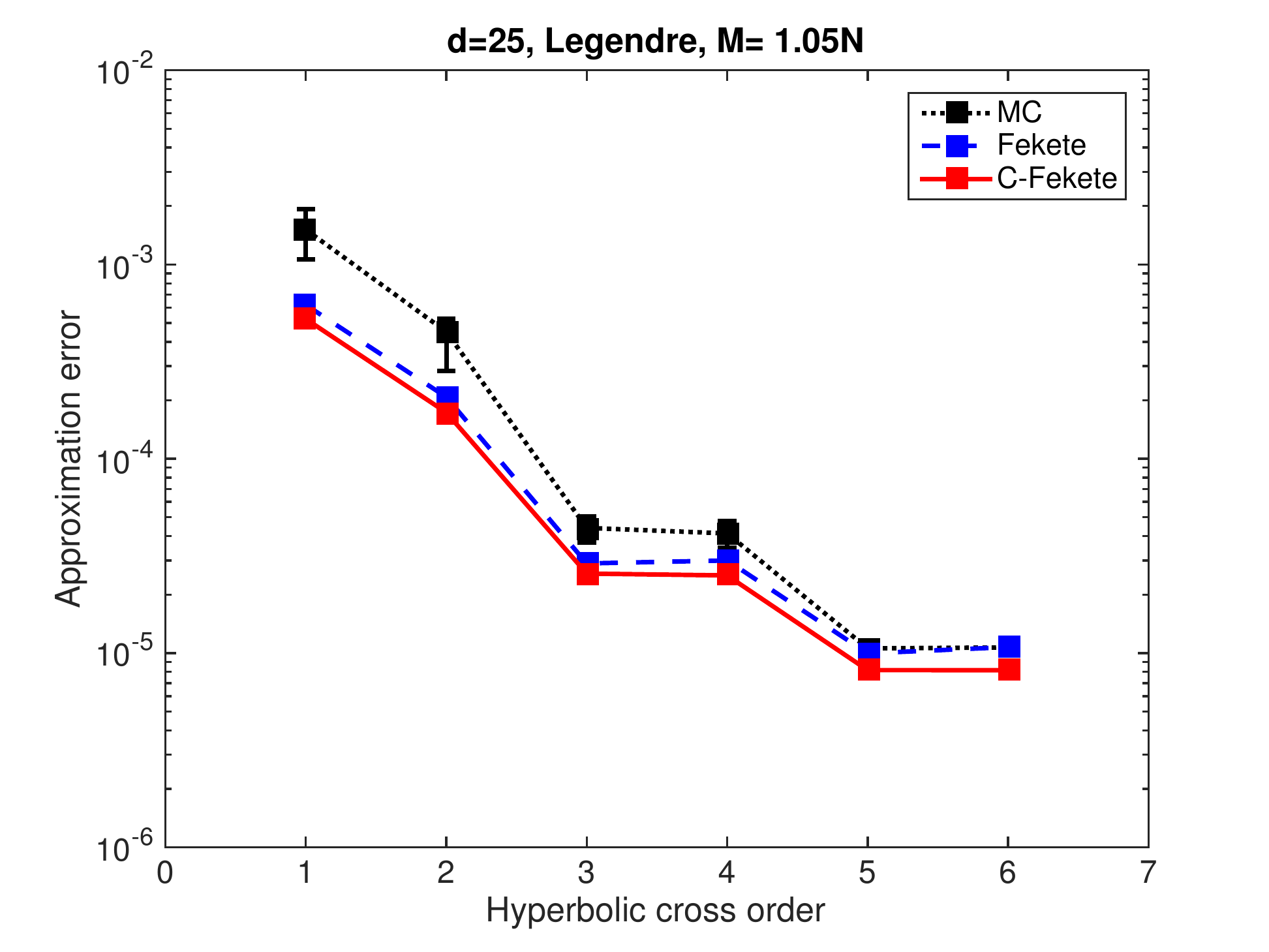}
\end{center}
  \caption{Approximation error against polynomial degree. Legendre approximation of the  diffusion equation. Left: $d=10$; Right: $d=25$.
      \label{fig:d25_Le_ell}
    }
\end{figure}

\appendix
\section{Proof of Theorem \ref{thm:nd-optimal}}\label{thm:nd-optimal-proof}
We first show the equivalence of the relations \eqref{eq:nd-optimal}. Recall the fact that for a matrix $\bs{A} \in \R^{N \times N}$ having columns $\bs{a}_i \in \R^N$, $i=1, \ldots, N$, then
\begin{align}\label{eq:det-col-inequality}
  \left| \det \bs{A} \right| \leq \prod_{i=1}^N \left\| \bs{a}_i \right\|_2,
\end{align}
with equality if and only if $\bs{a}_i$ are pairwise orthogonal. Now assume that \eqref{eq:nd-optimal-det} holds with $A_N = \left\{y_1, \ldots, y_N \right\}$. The matrix $\bs{V}(A_N, Q)$ is comprised of rows
\begin{align}\label{eq:V-rows-def}
\bs{V} &= \left( \begin{array}{c} \bs{\psi}^T(y_1) \\ \bs{\psi}^T(y_2) \\ \vdots \\ \bs{\psi}^T(y_N) \end{array}\right), & \bs{\psi}^T(y_j) &= \left( \frac{\psi_1(y_j)}{K_\Lambda(y_j)}, \; \frac{\psi_2(y_j)}{K_\Lambda(y_j)}, \ldots, \frac{\psi_N(y_j)}{K_\Lambda(y_j)}\right)^T.
\end{align}
We have
\begin{align}\label{eq:psi-row-orthonormal}
  \left\| \bs{\psi}(y_j)\right\|_2 = \sqrt{\sum_{i=1}^N \frac{\psi_i^2(y_j)}{K_\Lambda}} = 1,
\end{align}
for all $j$. Thus,
\begin{align*}
  \left| \det \bs{V}^T \right| = \left| \det \bs{V} \right| \stackrel{\eqref{eq:nd-optimal-det}}{=} 1 = \prod_{j=1}^N \left\| \bs{\psi}(y_j) \right\|_2
\end{align*}
thus showing equality in \eqref{eq:det-col-inequality} for $\bs{A} = \bs{V}^T$. Therefore, the $\bs{\psi}(y_j)$ are pairwise orthogonal, and are furthermore orthonormal because of \eqref{eq:psi-row-orthonormal}. A square matrix $\bs{V}$ with pairwise orthonormal rows is an orthogonal matrix, and hence has all of its singular values equal to 1. Therefore, $\kappa(\bs{V}) = 1$, showing \eqref{eq:nd-optimal-cond}.

Now assume that \eqref{eq:nd-optimal-cond} holds. Then $\bs{V}$, being a square matrix, is an orthogonal matrix. Hence,
\begin{align*}
  \left| \det \bs{V} \right|^2 = \left| \det \bs{V}^T \bs{V} \right| = \left| \det \bs{I} \right| = 1,
\end{align*}
showing that $\left| \det \bs{V} \right| = 1$, and hence \eqref{eq:nd-optimal-det}. This completes the proof of the equivalence of relations \eqref{eq:nd-optimal}.

%We require the following lemma:
%\begin{lemma}
%  Let $\bs{A} \in \R^{m \times N}$ with $m \leq N$ with each columns having $\ell^2$ norm equal to 1. Then with the rectangular determinant as in Definition \ref{def:det},
%  \begin{align*}
%    \left|\det \bs{A}\right| \leq 1
%  \end{align*}
%\end{lemma}
Now let $A_N = \left\{ y_1, \ldots, y_N \right\}$ be a set that satisfies either (hence, both) of the conditions \eqref{eq:nd-optimal}, and choose $y_1^{F\ast} \in A_N$; we further assume without loss that $y_1^{F\ast} = y_1$. (If not, then relabel the elements in $A_N$.) We proceed by induction, showing that for each $n=1, \ldots, N-1$, the greedy optimization \eqref{eq:Q-greedy-F} has a branch yielding the new point $y_{n+1}^{F\ast} = y_{n+1}$. Note that the initialization step of induction is true by assumption. We will show the induction step is true by noting that \eqref{eq:Q-greedy-F} varies a single $\R^N$ vector in an attempt to maximize an $(n+1)$-dimensional volume spanned by unit vectors. One way to maximize this volume is by choosing the vectors to correspond to the set $A_N$, which makes the volume an $(n+1)$-dimensional orthotope. The details are as follows.

Since conditions \eqref{eq:nd-optimal} are true, the vectors $\bs{\psi}(y_j)$, $j=1, \ldots, N$ are pairwise orthonormal. For some $n \geq 1$, assume $y_j^{F\ast} = y_j$ for $1 \leq j \leq n$. The iteration \eqref{eq:Q-greedy-F} at index $n \geq 1$ takes the form,
\begin{align*}
  y_{n+1}^{F\ast} &= \argmax_{y \in \Gamma} \left| \det\bs{V}\left(A_{n\cup y}^{F\ast}, Q\right) \bs{V}^T\left(A_{n\cup y}^{F\ast}, Q\right) \right| %\leq \argmax_{y \in \Gamma}, %\left\| \det %\left( \begin{array}{c}
\end{align*}
where
\begin{align*}
\bs{V}\left(A_{n\cup y}^{F\ast}, Q\right) = \left(\begin{array}{c} \bs{\psi}^T(y_1^{F\ast}) \\ \vdots \\ \bs{\psi}^T(y^{F\ast}_n)  \\ \bs{\psi}^T(y) \end{array}\right).
\end{align*}
Since $y_j^{F\ast} = y_j$ for $1 \leq j \leq n$, then $\bs{\psi}^T(y^{F\ast}_j)$ for $1 \leq j \leq n$ are pairwise orthonormal. This allows us to easily compute the $Q R$ decomposition of $\bs{V}^T$:
\begin{align*}
\bs{V}^T = \bs{Q} \bs{R} = \left( \begin{array}{ccccc} \bs{\psi}(y^{F\ast}_1) & \bs{\psi}(y^{F\ast}_2) & \cdots & \bs{\psi}(y^{F\ast}_n) & \widetilde{\bs{\psi}}(y) \end{array} \right) \left( \begin{array}{cc} \bs{I}_n & \bs{b} \\ \bs{0} & r \end{array}\right),
\end{align*}
%where
%\begin{align*}
%\bs{V}\left(A_{n\cup y}^{F\ast}, Q\right) \bs{V}^T\left(A_{n\cup y}^{F\ast}, Q\right) = \left( \begin{array}{cc} \bs{I}_n & \bs{b} \\ \bs{b}^T & 1 \end{array}\right),
%\end{align*}
where $\bs{I}_n$ is the $n\times n$ identity matrix, the vector $\bs{b} \in \R^{n}$ has entries
\begin{align*}
  (b)_j = \bs{\psi}^T\left(y_j^{F\ast}\right) \bs{\psi}\left(y\right),
\end{align*}
and $r = \sqrt{1 - \|\bs{b} \|_2^2}$. Note by Bessel's inequality that $\|\bs{b}\|_2 \leq 1$, and hence $r \in [0, 1]$. When $r > 0$, the last column of $\bs{Q}$ is
\begin{align*}
  \widetilde{\bs{\psi}}(y) = \frac{1}{r} \left[ \bs{\psi}(y) - \sum_{j=1}^n b_j \bs{\psi}(y_j^{F\ast}) \right],
\end{align*}
and when $r = 0$ we let $\widetilde{\bs{\psi}}(y)$ be any unit vector. Since the matrix $\bs{Q}$ is orthogonal, we have
\begin{align*}
\left| \det\bs{V}\left(A_{n\cup y}^{F\ast}, Q\right) \bs{V}^T\left(A_{n\cup y}^{F\ast}, Q\right) \right| = \left| \det(\bs{R}^T \bs{Q}^T \bs{Q} \bs{R}^T) \right| = \left| \det(\bs{R}^T \bs{R}) \right| = r^2.
\end{align*}
This is maximized when $r = 1$, which happens when $\bs{\psi}(y)$ is orthonormal to $\bs{\psi}(y_j^{F\ast})$ for $1 \leq j \leq n$. This is achievable by setting $y \gets y_{n+1}$ since $\bs{\psi}(y_{n+1})$ is a vector satisfying the desired orthonormality condition. Thus one maximizer of \eqref{eq:Q-greedy-F} is $y_{n+1}^{F\ast} = y_{n+1}$. This completes the inductive step, showing that $y^{F\ast}_{j} = y_j$ for $1 \leq j \leq N$.

Nearly the same argument shows that the greedy iteration \eqref{eq:Q-greedy-C} has a solution branch equal to $A_N$; we omit the proof.

\section{Proof of Lemma \ref{lemma:1d}}\label{lemma:1d-proof}
The results of this Lemma are essentially well-known. (E.g., historically \cite{szego_orthogonal_1975,freud_orthogonal_1971} or \cite{narayan2017} for a modern compilation of these results.) However, these results are scattered so we provide a proof here in order to be self-contained.

There are 5 enumerated statements in Lemma \ref{lemma:1d}. In this univariate $d=1$ case, our orthonormal polynomials $\varphi_n(z)$ satisfy a three-term recurrence, and the Christoffel-Darboux relation,
\begin{align}\nonumber
  z \varphi_n(z) &= \sqrt{b_n} \varphi_{n-1}(z) + a_n \varphi_n(z) + \sqrt{b_{n+1}} \varphi_{n+1}(z), \\\label{eq:cd}
  \sum_{j=0}^{n-1} \varphi_j(x) \varphi_j(z) &= \sqrt{b_n} \frac{ \varphi_n(x) \varphi_{n-1}(z) - \varphi_n(z) \varphi_{n-1}(x)}{x - z} \\
                                             &= \frac{\sqrt{b_n}}{\varphi_{n-1}(x) \varphi_{n-1}(z)} \left( \frac{r_n(x) - r_n(z)}{x - z}\right),
\end{align}
where the constants $a_n$ and $b_n$ depend on the polynomial moments of $\rho$. We define the set
\begin{align*}
  A_N(y) = r_N^{-1}\left( r_N(y) \right),
\end{align*}
where $r_N(y) \in \R$ since by assumption $y \not\in \varphi_{N-1}^{-1}(0)$. The function $r_N$ is meromoprhic with $N-1$ distinct poles on $\R$, and a straightforward computation shows that $r_N'(z)$ is continuous and positive everywhere except at its poles. This, coupled with the fact that $\lim_{z \rightarrow \pm \infty} r_N(z) = \pm \infty$ shows that the set $A_N(y)$ defined above consists of $N$ distinct points on $\R$. We label these points as $y_j$:
\begin{align*}
  A_N(y) = \left\{ y_1, \ldots, y_N \right\}.
\end{align*}
By definition, $r_N(y_j)$ is the same number for any $1 \leq j \leq N$. We now consider the matrix $\bs{V}\left(A_N(y); Q\right)$, whose rows are $\bs{\psi}^T(y_j)$, which is defined in \eqref{eq:V-rows-def}. Note that $\bs{\psi}(y_j)$ is a unit vector for each $j$. We have, for $j \neq k$,
\begin{align*}
  \bs{\psi}^T(y_j) \bs{\psi}(y_k) &= \frac{1}{\sqrt{K_\Lambda(y_j) K_\Lambda(y_k)}} \sum_{j=0}^{N-1} \varphi_j(y_j) \varphi_j(y_k) \\
                                  &\stackrel{\eqref{eq:cd}}{=} \frac{\sqrt{b_N}}{\varphi_{N-1}(y_j) \varphi_{N-1}(y_k) \sqrt{K_\Lambda(y_j) K_\Lambda(y_k)}} \left( \frac{r_N(y_j) - r_N(y_k)}{y_j - y_k}\right) = 0,
\end{align*}
where the last equality holds since $r_N(y_j) = r_N(y_k)$ with $y_j \neq y_k$. Thus, $\bs{V}$ is a square matrix with orthonormal rows; therefore it is an orthogonal matrix, and has modulus determinant 1 and condition number 1. Therefore, $A_N(y)$ satisfies \eqref{eq:nd-optimal}, proving the first statement in the Lemma. The second statement, uniqueness of $A_N(y)$ is straightforward given the construction above. The third statement, defining $A_N(y)$ as level sets of $r_N$, is our explicit construction above.

To show the fourth statement, consider the matrix $\bs{V} = \bs{V}(A_N(y), Q)$, which we have already shown is an orthogonal matrix, and hence
\begin{align*}
  \bs{V}^T \bs{V} = \bs{I}_N
\end{align*}
For row and column indices $i$ and $j$, respectively, the componentwise equality above reads
\begin{align}\label{eq:quadrule}
  \sum_{q=1}^N \frac{1}{K_\Lambda(y_q)} \varphi_i(y_q) \varphi_j(y_q) = \delta_{i,j} = \int_\Gamma \varphi_{i-1}(z) \varphi_{j-1}(z) \rho(z) d z,
\end{align}
for $0 \leq i,j \leq N-1$, where the final equality is just orthonormality of the polynomials $\varphi_i$. This shows that the quadrature rule whose abscissae are collocated at $A_N(y)$ exactly integrates products $\varphi_{i} \varphi_{j}$ for $0 \leq i, j \leq N-1$. Let $p$ be an arbitrary polynomial of degree $2N-2$ or less. Euclidean division of this polyomial by $\varphi_{N-1}$ yields
\begin{align*}
  p = \varphi_{N-1} q + r,
\end{align*}
where $q$ is a polynomial of degree $N-1$ or less, and $r$ is a polynomial of degree $N-2$ or less. We now use the fact that $\left\{ \varphi_j\right\}_{j=0}^{N-1}$ is a basis for polynomials of degree $N-1$ and less, and so there exist constants $c_j$ and $d_j$ such that
\begin{align*}
  q(z) &= \sum_{j=0}^{N-1} c_j \varphi_j(z), &
  r(z) &= \sum_{j=0}^{N-1} d_j \varphi_j(z),
\end{align*}
with $d_{N-1} = 0$. Since $\rho$ is a probability density, then $\varphi_0(z) = 1$, implying
\begin{align*}
  p(z) = \sum_{j=0}^{N-1} c_j \varphi_j(z) \varphi_{N-1}(z) + \sum_{j=0}^{N-1} d_j \varphi_j(z) \varphi_0(z).
\end{align*}
Since the quadrature rule in \eqref{eq:quadrule} is linear and can exactly integrate products $\varphi_i \varphi_j$, then it can exactly integrate $p$, which equals a sum of such products. This shows statement four of the Lemma.

The final statement is straightforward: if $y \in \varphi_N^{-1}(0)$, then $r_N(y) = 0$, and so $A_N = r_N^{-1}(0)$. The zero level set $r_N^{-1}(0)$ coincides with the zero level set $\varphi_{N}^{-1}(0)$ since the roots of $\varphi_{N-1}$ and the roots of $\varphi_{N}$ are disjoint sets. This shows that $A_N(y)$ is the zero set of $\varphi_N$; therefore, these points are the Gaussian quadrature nodes, proving statement five of the Lemma.

\section{Proof of Thereom \ref{thm:1d-optimal}}\label{thm:1d-optimal-proof}
This proof is essentially a combination of Lemma \ref{lemma:1d} and Thereom \ref{thm:nd-optimal}. By Lemma \ref{lemma:1d}, the set $A_N$ has determinant modulus and condition number 1, satisfying \eqref{eq:nd-optimal}. Thus, Theorem \ref{thm:nd-optimal} guarantees that both greedy algorithms \eqref{eq:Q-greedy} produce the same optimal set $A_N$. Since the set $A_N$ containing $y$ is unique, then this is true regardless of which solution branches are taken during the greedy iterations \eqref{eq:Q-greedy}.

\bibliographystyle{plain}
\bibliography{OSAUQ}

\end{document}